\documentclass[12pt]{amsart}
\usepackage{amssymb,amsmath,amscd,graphicx,
latexsym,amsthm}
\usepackage{amssymb,latexsym,eufrak,amsmath,amscd,graphicx}
  \usepackage[all]{xy}
  \usepackage{pdfsync}
\theoremstyle{plain}
\newtheorem{theorem}{Theorem}
\newtheorem{proposition}[theorem]{Proposition}
\newtheorem{corollary}[theorem]{Corollary}
\newtheorem{lemma}[theorem]{Lemma}

\newtheorem{proposition.definition}[theorem]{Proposition/Definition}

\newtheorem{theoremalpha}{Theorem}
\newtheorem{corollaryalpha}[theoremalpha]{Corollary}

\theoremstyle{definition}
\newtheorem{definition}[theorem]{Definition}

\newtheorem{remark}[theorem]{Remark}

\newtheorem{problem}[theorem]{Problem}
\newtheorem{conjecture}[theorem]{Conjecture}

\newtheorem{NC}{}

\newcommand{\lra}{\longrightarrow}

\newcommand{\noi}{\noindent}
\newcommand{\PP}{\mathbf{P}}

\newcommand{\RR}{\mathbf{R}}

\newcommand{\ZZ}{\mathbf{Z}}
\newcommand{\CC}{\mathbf{C}}

\newcommand{\OO}{\mathcal{O}}

\newcommand{\II}{\mathcal{I}}

\newcommand{\eps}{\varepsilon}

\newcommand{\rk} {\text{rank }}

\newcommand{\Image}{\textnormal{Im}}
\newcommand{\HH}[3]{H^{{#1}} \big( {#2} , {#3}
\big) }
\newcommand{\HHH}[3]{H^{{#1}} \Big( {#2} \, , \, {#3}
\Big) }
\newcommand{\hh}[3]{h^{{#1}} \big( {#2} , {#3}
\big) }

\newcommand{\tn}[1]{\textnormal{#1}}

\newcommand{\pr}{\prime}

\newcommand{\lin}{\equiv_{\text{lin}}}

\newcommand{\Linser}[1]{| \mspace{1.5mu} {#1}
\mspace{1.5mu} |}
\newcommand{\linser}[1]{\Linser{  {#1}  }}

\newcommand{\Tor}{\textnormal{Tor}}
\newcommand{\Sym}{\textnormal{Sym}}
\newcommand{\Hh}[2]{H^{{#1}}({#2})}
\newcommand{\ev}{\textnormal{ev}}

\newcommand{\ol}[1]{\overline{#1}}

\numberwithin{equation}{section}
\numberwithin{theorem}{section}

\begin{document}

\title{Asymptotic Syzygies of Algebraic Varieties}

 \author{Lawrence Ein}
  \address{Department of Mathematics, University Illinois at Chicago, 851 South Morgan St., Chicago, IL  60607}
 \email{{\tt ein@uic.edu}}
 \thanks{Research of the first author partially supported by NSF grant DMS-1001336.}

 \author{Robert Lazarsfeld}
  \address{Department of Mathematics, University of Michigan, Ann Arbor, MI
   48109}
 \email{{\tt rlaz@umich.edu}}
 \thanks{Research of the second author partially supported by NSF grant DMS-0652845.}

\maketitle
\setlength{\parskip}{0 pt}
\setcounter{tocdepth}{1}

\tableofcontents
\setlength{\parskip}{.20 in}
 
 \section{Introduction}

 The purpose of this paper is to study the asymptotic behavior of the syzygies of a  smooth projective variety  $X$ as the positivity of the embedding line bundle grows.  We prove that as least as far as grading is concerned,  the minimal resolution of the ideal of $X$ has a surprisingly uniform asymptotic shape: roughly speaking, generators   eventually appear in almost all degrees permitted by  Castelnuovo-Mumford regularity. This suggests in particular that a widely-accepted intuition  derived from the case of curves -- namely that syzygies become simpler as the degree  of the embedding increases --  may have been misleading. For Veronese embeddings of projective space, we give an effective statement that  in some cases is optimal, and conjecturally always is so. Finally, we propose a number of questions and open problems concerning asymptotic syzygies of higher-dimensional varieties. 
  
 Turning to details,  let $X$ be a smooth projective variety of dimension $n$ defined over an algebraically closed field ${k}$, and let $L$ be a very ample divisor on $X$. Thus $L$ defines an embedding
 \[   X\  \subseteq \  \PP^{r}  = \ \PP \HH{0}{X}{\OO_X(L)}, \]
 where $r = r(L) = \hh{0}{X}{\OO_X(L)} - 1$.
We propose to study   the syzygies of $X$ in $\PP^r$ when $L$ is very positive. To this end, let $S = \Sym \, \HH{0}{X}{\OO_X(L)}$ be the homogeneous coordinate ring of $\PP^r$, and let 
 \[ R \ = \  R(L) \ = \ \oplus \, \HH{0}{X}{\OO_X(mL)} \] be the graded ring associated to $L$, viewed as an $S$-module. Then $R$ has a minimal graded free resolution $E_\bullet = E_\bullet(X;L)$:
 \[
 \xymatrix{
0 & R \ar[l]& \oplus S(-a_{0,j}) \ar[l]  \ar@{=}[d]& \oplus S(-a_{1,j})  \ar[l] \ar@{=}[d]& \ar[l] \ldots & \ar[l] \oplus S(-a_{r,j}) \ar[l] \ar@{=}[d] &  \ar[l]0 . \\ 
&  & E_0 & E_1 & &E_r }
 \]
 Starting with the pioneering work \cite{Kosz1} of Green, there has been considerable interest in understanding what one can say about the degrees $a_{p,j}$  of the generators  of the  $p^{\text{th}}$ module of syzygies of $R$. Note that if $L$ is normally generated -- i.e. if $E_0 = S $ -- then $E_\bullet$ determines a resolution of the homogeneous ideal of $I_X$ of $X$. So in the presence of normal generation, the question amounts to asking about the degrees of the equations defining $X$ and the syzygies among them. Conditions for projective normality and degrees of defining equations were studied over the years by many mathematicians, including Castelnuovo \cite{Castelnuovo}, Mumford \cite{Mumford1}, \cite{Mumford2},  and Bombieri and his school   \cite{Bombieri}, \cite{Catanese}, \cite{Catanese2}.  Green's idea was that one should see the classical results  as the beginning of a   more general picture for higher syzygies.

Specifically, Green considered the case when $X$ is a curve of genus $g$, and $L = L_d$ is a divisor of degree $d \ge 2g+1$.  Write $r_d = r(L_d)$, so that  \[ r_d \ = \ d-g.\] It is classical that  $L_d$ is normally generated  and that furthermore $E_\bullet$ has length $r_d-1$. Green proved that  if $1 \le p \le d - 2g -1$, then all the generators of $E_p$ appear in the lowest possible degree $p +1$, i.e.
\[     E_p \ = \ \oplus\,  S(-p-1) \ \ \text{ for } \  p \in\big  [1, d - 2g -1\big]. \] Very concretely, this means that if $d - 2g - 1 \ge 1$, then the homogeneous ideal $I_X$ of $X$ is generated by quadrics;\footnote{This was also known classically.}  if $d - 2g - 1 \ge 2$, then the module of syzygies among quadratic generators $q_\alpha \in I_X$ is spanned by relations of the form 
\[  \sum \ell_\alpha \cdot q_\alpha \ = \ 0\]
where the $\ell_\alpha$ are linear forms; and so on.\footnote{The conclusion is summarized by saying that $L_d$ satisfies Property  ($N_p$) when  $p \le d - 2g-1$.   It is also known that $E_p$ can acquire generators of degree $ p + 2$ when $d -2g \le p \le d-g-1$, and that in fact this always happens when $d \ge 3g$. } 
Observing that $d - 2g -1 =(r_d - 1) - g$, one can view Green's theorem as asserting that as $d$ grows, all but a fixed number of the syzygy modules of $I_X$ are as simple as possible in terms of degrees of generators. 

After \cite{Kosz1}, several results appeared giving analogous statements about the higher syzygies of  varieties of larger dimension. For example, Green \cite{Kosz2} proved that when $X =  \CC\PP^n$ and $L_d \in \linser{\OO_{\PP^n}(d)}$,   one has
\begin{equation}
E_p(X;L_d) \ = \ \oplus \,  S(-p-1)  \ \ \text{for } 1 \le p \le d. \tag{*}
\end{equation}
This was generalized by the authors in \cite{SAD} to an arbitrary smooth complex projective variety $X$ of dimension $n$. Specifically, if 
\[  L_d \ = \  K_X + (n+1 + d)B    \]
for some very ample divisor $B$ on $X$, 
then (*) continues to hold. For abelian varieties, Pareschi \cite{Pareschi} established analogous bounds that do not depend on the dimension. These and similar results have led to the philosophy -- espoused among others by the current authors --  that syzygies become simpler as the positivity of the embedding line bundle grows. 

The starting point of the present paper was the remark that when $n = \dim X \ge 2$, the known results just summarized describe only a small fraction of the syzygies of $X$. In fact, for divisors $L_d$ of the sort appearing above, one has
\[      r(L_d) \ = \ O(d^n). \]
In particular, the length of the resolution $E_\bullet(X, L_d)$ grows like a polynomial of degree $n$ in $d$. So in dimensions two and higher, statements such as (*) that are linear in $d$ ignore most of the syzygy modules that occur. It therefore seemed interesting to ask whether one can say anything about the overall shape of  $E_\bullet(X, L_d)$ for $d \gg 0$.\footnote{A question along these lines was raised by Green in \cite[Problem 5.13]{Kosz1}, and  by the authors in \cite[Problem 4.4]{SAD}.} It turns out that -- at least as far as grading is concerned -- there is a surprisingly clean asymptotic picture, which moreover suggests that the intuition derived from curves may have been misleading. 

In order to state our results we require one piece of additional notation. With $X$ and $L$ as above, define (for the purposes of this Introduction): 
\[
K_{p,q} \big ( X; L \big) \ = \ \Big \{
\parbox{2.3in}{\begin{center} minimal generators of $E_p(X;L)$ of degree $p + q$\end{center}}
\Big  \}.  
\]
  Thus $K_{p,q}(X;L)$ is a finite-dimensional vector space, and
\[  E_p(X;L) \ = \  \underset{q}{\bigoplus}   \ K_{p,q}(X;L) \, \otimes _k \,  S(-p-q). \]
So for example, if $L$ is normally generated,  then 
\[ K_{1,q}(X;L) \ = \ \big \{ \text{degree $(q+1)$ generators of the homogenous ideal $I_X$} \big \}.
\]
  Or again,  to say that 
$ E_p(X;L)  = \oplus \, S(-p-1)$ means that $K_{p,q}(X;L) = 0 $ for $q \ne 1$.  

Our goal is to study the asymptotic behavior of these groups for increasingly positive line bundles on $X$. Fix an ample divisor $A$, and an arbitrary divisor $P$, and put
\[
L_d \ = \  dA + P.
\]
We will always assume that $d$ is sufficiently large so that $L_d$ is very ample, and we note that   \[ r_d \ =_{\text{def}} \ r(L_d)\] is given for large $d$ by a polynomial of degree $n$ in $d$. Concerning the syzygies of $L_d$ for $d \gg 0$, the picture is framed by three general facts. First, as is well known:
\begin{equation}
\text{$ L_d$ is normally generated provided that  $d$ is sufficiently large.}
\tag{1}\end{equation} Next, 
elementary considerations of Castelnuovo-Mumford regularity show that if $d \gg 0 $, then
\begin{equation}   K_{p,q}(X; L_d) \ = \ 0 \ \text{ for } \ \ q > n + 1. \tag{2}
\end{equation}
Finally, results  of Green \cite{Kosz1}, Schreyer \cite{Schreyer} and Ottaviani--Paoletti \cite{OP} imply that   $K_{p,n+1}(X;L_d)$ can be non-vanishing for large $d$ only when \[
  p_g(X)  \ =_{\text{def}}\ \hh{0}{X}{\OO_X(K_X)}\ \ne \ 0, \] in which case
\begin{equation}
K_{p,n+1}(X;L_d) \ \ne \ 0 \ \Longleftrightarrow \   p \, \in \, \big [ r_d -n + 1 -p_g   , r_d - n \big]. \tag{3}
\end{equation}
So to determine the grading of $E_\bullet(X;L_d)$ for $d \gg 0$,  the issue is to understand
which of the groups $K_{p,q}(X;L_d)$ are non-vanishing for $1 \le p \le r_d$ and  $1 \le q \le n$.

Our main result states that asymptotically in $d$, the vector space in question is non-zero for essentially all   $p \in [1, r_d]$ and $q \in [1, n]$. 
\begin{theoremalpha} \label{Main.Thm.Intro}
Let $X$ and $L_d$ be as in the previous paragraph, and fix an index $1 \le q \le n$.  There exist constants $C_1 , C_2> 0$ with the property that if $d$ is sufficiently large, then  \[  K_{p, q}(X; L_d) \ \ne \ 0  \]
for every value of $p$ satisfying
 \[  
C_1\cdot d^{q-1} \ \le \ p \ \le \ r_d - C_2 \cdot  d^{n-1} .  \]
If moreover 
$\HH{i}{X}{\OO_X} = 0$ for $0 < i < n$, then    the non-vanishing holds in the range
 \[  
C_1\cdot d^{q-1} \ \le \ p \ \le \ r_d - C_2 \cdot  d^{n -q} .  \]

\end{theoremalpha}
\noi   
Thus $K_{p,2}(X;L_d)$ becomes non-zero at least linearly in $d$; $K_{p,3}(X;L_d)$ becomes non-zero at least quadratically in $d$; and so on.\footnote{We conjecture that conversely
\[
K_{p,q}(X; L_d) \ = \ 0 \ \ \text{ for  } \  \ p \le O(d^{q-1}).\] 
  See \S \ref{Conjs.Open.Problems.Section}
 for further discussion.}
Furthermore the corresponding strands of the resolution of $X$ have almost the maximal possible length $r(L_d)$. We remark that it was established by Ottaviani and Paoletti in \cite{OP} that 
\[ K_{p,2}(\PP^2, \OO_{\PP^2 }(d)) \ \ne  \ 0 \ \  \text{for } \ 3d -2 \, \le \, p \, \le \, r_d-2, \]
and Eisenbud et.\ al.\ \cite{EGHP} proved a non-vanishing  on other surfaces $X$ implying (at least if $p_g(X) = 0$) that $K_{p,2}(X, L_d) \ne 0$ for $O(d) \le p \le O(d^2)$.\footnote{The authors just cited were mainly concerned with the failure of property ($N_p$), so they did not explicitly address the question of giving a range of $p$ for which $K_{p,2}(X; L_d)\ne 0$. However the existence of an interval of the indicated shape follows from their results thanks to observations (2) and (3) above. We remark also that the
computations  in  \cite{OP} yield  for free that
\[ 
K_{p,2}(\PP^n, \OO_{\PP^n}(d) ) \ \ne \ 0 \ \ \text{ for \ }  O(d) \le p \le O(d^2). \]
But the non-vanishing of this group for $O(d) \le p \le O(d^n)$ already seems to be non-trivial when $n \ge 3$.  By contrast, the fact that $K_{p,1}(X; L_d) \ne 0$ for $1 \le p \le O(d^n)$ can be seen directly in the case at hand by observing that $X$ lies on scrolls of codimension $r_d - O(d^{n-1})$ in $\PP^{r_d}$.
}
These results were important in guiding our thinking about these questions.

The theorem implies  in particular that for every index $1 \le q \le n$, 
there exist functions
$p_-(d) $ and $ p_+(d)$, with
\[
\lim_{d\to \infty} \frac{ p_-(d)}{r_d} \ = \ 0\ \ , \ \ \lim_{d\to \infty} \frac{ p_+(d)}{r_d} \ = \ 1,  
\]
such  that  if $d$ is sufficiently large  then
\[  K_{p, q}(X; L_d) \ \ne \ 0 \ \ \text{ for all} \ \ p \in \big [\,   p_-(d) \, , \, p_+(d)\,  \big ].
\]
It is amusing to visualize this conclusion in  terms of the {Betti diagram} of $X$. By definition, this is the $(n+1) \times r_d$ matrix whose $(q,p)^{\text{th}}$ entry records the dimension of $ K_{p,q}(X; L_d)$. Now  imagine constructing a ``normalized Betti diagram" by horizontally rescaling  the Betti table so that it fits into a  rectangle with $n+1$ rows and fixed width (independent of $d$).  Then the non-zero entries occupy a region which as $d \to \infty$   entirely  fills up the first $n$ rows.  Similarly, the result of Green et.\  al.\  quoted above -- which asserts that $K_{p, n+1}(X; L_d) \ne 0$ for only a fixed number of values of $p$ --  implies that the possibly non-zero entries in the bottom row of the normalized diagram are concentrated in a segment whose length approaches $0$ as $d \to \infty$.  

A particularly interesting case of these questions  arises when  $X = \PP^n$. Here it is natural to take $L_d \in \linser{\OO_{\PP^n}(d)}$, so that  we are studying the asymptotic behavior of the syzygies of the $d$-fold  Veronese embedding of $\PP^n$. These Veronese syzygies  have been the focus of a great deal of attention, e.g.\ \cite{OP}, \cite{Rubei}, \cite{BCR}, \cite{BCR2}. We establish an effective    result that at least in some cases  is best-possible. 
\begin{theoremalpha} \label{Veronese.Syzygies.Intro}      
Fix an index $1 \le q \le n$. If $d$ is sufficiently large, then  
\[ K_{p,q}\big(\PP^n; \OO_{\PP}(d)\big) \ \ne \ 0\]
whenever
\[
 \binom{d+q}{q} \, - \,  \binom{d-1}{q} \, - \, q   \ \le \  p \ \le \ \binom{d +n  }{n} \, - \, \binom{d+n -q}{n-q}\, + \,  \binom{n}{n-q}  - q-1. \]
\end{theoremalpha} 
\noi (Although we will not give all the calculations, one can use our arguments to show that in fact the assertion holds as soon as $d \ge n +1$.)
When $q = n$,  the Koszul cohomology group in question vanishes if $p$ lies outside the stated range. We conjecture that the inequality is actually optimal for every  $q \in [1,n]$ and $d \ge q+1$.\footnote{We have been informed by J. Weyman that he independently obtained  the non-vanishing of these Veronese syzygies in characteristic $0$, and   likewise conjectures the optimality of the statement.}

In characteristic zero, Theorem \ref{Veronese.Syzygies.Intro} implies a rather clean   non-vanishing statement for the decomposition of Veronese syzygies into irreducible representations of the general linear group. Specifically, denote by $\mathbf{K}_{p,q}(d)$ the functor  that associates to a complex vector space $U$ the group\[  K_{p,q}\big(\PP(U), \OO_{\PP(U)}(d)\big).\] As in \cite{Rubei} or \cite{Snowden} one can write
\[
 \mathbf{K}_{p,q}(d) \ = \ \bigoplus_{\lambda \, \vdash \, (p+q)d} M_{\lambda}  \otimes_{\CC} \mathbf{S}_\lambda,  
\]
where $\mathbf{S}_{\lambda}$ is the Schur functor associated to the partition $\lambda$, and $M_{\lambda}= M_{\lambda}(p,q;d)$ is a finite-dimensional vector space governing the multiplicity of $\bf{S}_\lambda$ in $ \mathbf{K}_{p,q}(d)$.  It was established by Rubei \cite{Rubei} that once $n +1 = \dim U\ge p+1$, the vanishing or non-vanishing of $\mathbf{K}_{p,q}(d)(U)$ is independent of $U$. One can view the theorem   of Green cited above as asserting that  if $q \ge 2$ then $\mathbf{K}_{p,q}(d) = \mathbf{0}$ for   $p \le d$. 
\begin{corollaryalpha} Fix $q \ge 1$ and  $d \ge q+1$, and 
assume that 
\[
p \ \ge \ \binom{d+q}{q} \, - \, \binom{d-1}{q} \, - \, q.
\]
Then $\mathbf{K}_{p,q}(d) \ne \mathbf{0}$.\end{corollaryalpha}
\noi In other words, if $p$ satisfies the stated inequality then   some $\bf{S}_\lambda$ appears non-trivially in $\mathbf{K}_{p,q}(d)$. Conversely,   it seems possible as above that one might expect  $\mathbf{K}_{p,q}(d) = \mathbf{0}$ when $p < \binom{d+q}{q} - \binom{d-1}{q} - q$: for $q = 2$ this is (part of) a conjecture of Ottaviani and Paoletti \cite{OP}. 
 
Returning to an arbitrary smooth projective variety $X$, fix a divisor $B$ on $X$ and consider
\[  
R ( B; L_d) \ = \ \oplus \, \HH{0}{X}{\OO_X(B + mL_d)}.
\] 
This is a graded module over $S = \Sym \, \HH{0}{X}{\OO_X(L_d)}$, and one can ask about its syzygies $K_{p,q}\big(X, B; L_d\big)$. We work with this set-up in the body of the paper.   For a given $B$ one finds an asymptotic picture similar to the case $B = 0$ summarized in the preceding paragraphs.

Our proofs revolve around secant constructions. Let $\Lambda = \Lambda_d \subseteq \PP^{r_d}$ be a linear subspace of dimension $ s_d$  that meets $X$ in a scheme $ Z_d$.  Under a mild hypothesis that holds automatically when $d \gg 0$, such a plane  gives rise to a homomorphism
\[   \alpha_{\Lambda} \, : \, K_{s_d +2 -q,q}\big(L_d\big) \lra H^{q-1}\big(X, \II_{Z_d/X}(L_d)\big) ; \]
  the strategy of course is to choose $\Lambda$ in such a way that
\[ \HH{q-1}{X}{\II_{Z_d/X}( L_d)} \ne 0.  
\]  
In the case of curves, the fact that highly secant planes carry syzygies goes back  to \cite{GL1} and \cite{GL2}, and in \cite{EGHP} special secants of the sort appearing in our arguments were used to study the failure of Property ($N_p$) on surfaces. The essential issue here -- which was automatic in the earlier applications --  is to show that the homomorphism $\alpha_{\Lambda}$ is non-zero. Our main technical result sets up an induction on dimension by establishing that it suffices to check the corresponding statement for the embeddings of a fixed hypersurface $\ol{X} \subset X$. A priori, this yields a   $p^{\text{th}}$ syzygy for one specific value of $p$ (viz.\ $p = s_d + 2 - q$). However we can  enlarge the secant plane without changing how it meets the image of $\ol{X}$, and in this way we are able to produce non-zero classes in $K_{p,q}$ for a large range of $p$. 
 
Concerning the organization of the paper, after a quick review of the basics of Koszul cohomology, Section \ref{Secant.Constructions.Section} takes up secant constructions: the main technical results (Theorems \ref{First.Technical.Thm}
 and \ref{Second.Technical.Theorem}) appear there. The asymptotic non-vanishing theorem is established in \S \ref{Asympt.Nonvan.Section}, and completed in \S \ref{Kp1.Section} where the groups $K_{p,0}$ and $K_{p,1}$ are discussed. 
Section \ref{Veronese.Vars.Section} is devoted to  effective statements for Veronese embeddings of projective space.  Finally, in \S \ref{Conjs.Open.Problems.Section} we discuss some conjectures and open problems. There seems to be a lot left to learn about the syzygies of higher-dimensional varieties, and we hope that these questions may stimulate further work in this direction.
 
 We are grateful to D. Eisenbud, D. Erman, M. Fulger, W. Fulton, M. Hering, K. Lee, M. Musta\c t\v a, F. Schreyer, J. Sidman, V. Srinivas, and X. Zhou for helpful comments and suggestions. We are especially indebted  to Andrew Snowden for many valuable discussions, particularly concerning
 the material in \S \ref{Veronese.Vars.Section}.
 
%
%
 \section{Notation and Conventions}
 
\begin{NC}
We work throughout over an algebraically closed field $k$. A \textit{variety} is reduced and irreducible.\footnote{The non-vanishing of syzygies  is unaffected by extension of the base field. Therefore all of our results remain valid for a geometrically  irreducible non-singular  variety defined over an arbitrary field.}
\end{NC}

\begin{NC}
If $X$ is a variety, and $V$ is a finite-dimensional $k$-vector space, we write \[ V_X \ = \ V \otimes_k \OO_X
\]
for the trivial vector bundle on $X$ modeled on $V$. We denote by $\PP(V)$ the projective space of one-dimensional \textit{quotients} of $V$. 
\end{NC}

\begin{NC}
A \textit{divisor} on a variety $X$ is a Cartier divisor. As a compromise between the languages of line bundles and divisors,  we often use notation suggestive of bundles to refer to divisors. Given a divisor $L$ on $X$, we try to write \[ \HH{i}{X}{\OO_X(L)}\] to denote the cohomology groups of the corresponding locally free sheaf. However in the interests of compactness we sometimes use the shorthand $\HH{i}{X}{L}$ or $H^i(L)$ instead.
\end{NC}
 
 \begin{NC}
 When $X$ is smooth, $K_X$ denotes as usual a canonical divisor on $X$. 
 \end{NC}
 
 \begin{NC}
 Given a non-negative function $f(d)$ defined for sufficiently large positive integers $d$, we say that $f(d) = O(d^k)$ if
 \[
 \liminf \ \frac{f(d)}{d^k} \ \  \text{and} \ \ \limsup\  \frac{f(d)}{d^k }
 \]
 are finite and non-zero.
 \end{NC}

%
%
 
\section{Secant Constructions} \label{Secant.Constructions.Section}

Throughout this section, $X$ denotes an irreducible projective variety of dimension $n$, and $L$ is a very ample divisor on $X$. We fix once and for all a basepoint-free subspace $ V   \subseteq \HH{0}{X}{\OO_X(L)} $ of dimension $v$ that defines an embedding        
\[   X \ \subseteq \ \PP(V) 
\, = \, \PP^{v-1}. \]
In the sequel we will be mainly concerned with the situation where $V = \HH{0}{X}{\OO_X(L)}$ is the complete linear series determined by $L$, but in the interests of clarity we prefer not to limit ourselves to that case here. 

\subsection*{Review of Koszul Cohomology} We wish to study various syzygy modules associated to  $X$ in $\PP(V)$. Specifically, let 
$S = \Sym(V)$ be the homogeneous coordinate ring of $\PP(V)$. Next, fix a divisor $B$ on $X$, and put
\[ 
R \ = \ R(B;L) \ =_\text{def} \ \oplus_m \,  \HH{0}{X}{\OO_X(mL + B) }.
\]
This is naturally a graded $S$-module, and we  denote by 
\[E_\bullet \ = \  E_\bullet (X, B; V) \]
the minimal graded free resolution of $R$ over  $S$. We will often use some hopefully self-explanatory perturbations and abbreviations of this notation. Notably, when  $V$ is the complete space of sections $V = \HH{0}{X}{\OO_X(L)}$  of $L$ we write $E_\bullet(X,B;L)$, and when in addition $B = 0$ we put simply $E_\bullet(X;L)$. 

We recall two basic and well-known facts about these syzygy modules. The first is that the graded pieces of $E_\bullet$ are computed as Koszul cohomology groups. 
\begin{proposition.definition} Let 
$
K_{p,q}(X, B;V) $
denote the cohomology of the Koszul-type complex
\begin{equation} \label{KpqDef}
 \Lambda^{p+1} V \otimes \Hh{0}{B+ (q-1)L} \lra \Lambda^p V \otimes  \Hh{0}{ B + qL} \lra \Lambda^{p-1} V\otimes \Hh{0}{ B + (q+1)L}. \notag \end{equation}
Then
\[
K_{p,q}(X, B; V) \ = \ \Tor_p^S \big(R(B;L),S/S_+\big)_{p+q},\]
where $S_+ \subseteq  S$ denotes the irrelevant maximal ideal. In particular
\[  E_p(X,B;V) \ = \  \underset{q}{\bigoplus}   \ K_{p,q}(X,B;V) \, \otimes _k \,  S(-p-q). \]
\end{proposition.definition}
\noi We refer for example to \cite{Kosz1} or \cite{VBT} for the proof. As above, we will write $K_{p,q}(X, B;L)$ when $V = \HH{0}{X}{\OO_X(L)}$, and we omit mention of $B$ in the event that $B = 0$. For lack of a better term, we will refer to a class in $K_{p,q}(X, B;L)$ as a \textit{$p^{\text{th}}$ syzygy of weight $q$}. 

The second important fact for us is that  these Koszul cohomology groups are governed by the coherent cohomology of a vector bundle on $X$. Specifically, there is a natural evaluation map 
\[  \ev_V : V_X \, =_{\text{def}} \, V \otimes_k \OO_X \lra\OO_X(L),  \]
and we put $M_V = \ker e_V$. Thus $M_V$ is a vector bundle of rank $v-1$ sitting in the basic exact sequence
\begin{equation} \label{Seq.Def.M_V}
0 \lra M_V \lra V_X \lra \OO_X(L)\lra 0. 
\end{equation} 
When $V = \HH{0}{X}{\OO_X(L)}$ is the complete linear series associated to $L$, we write $M_L$ in place of $M_V$. 
\begin{proposition} \label{Non.Van.Wedge.p+1.Gives.Koszul}
Assume that
\begin{equation} \label{Van.Hypoth.Kosz.from.M}
\HH{i}{X}{\OO_X(B + mL)} \ = \ 0  \ \ \text{for } \ i > 0 \ \text{ and } \ m > 0. \end{equation}
Then for $q \ge 2$:
\[  K_{p,q}(X,B;V) \ = \ \HHH{1}{X}{\Lambda^{p+1} M_V \otimes \OO_X(B + (q-1)L)}. \]
If moreover $\HH{1}{X}{\OO_X(B)} = 0$, then the same statement holds also when $q = 1$. 
 \end{proposition}
\begin{proof}[Brief Sketch of Proof]
Starting with the Koszul complex determined by the evaluation  map $V_X \lra \OO(L)$, and twisting by $\OO_X(B + qL)$, one arrives at a complex $\mathcal{K}_\bullet$   of sheaves whose space of sections computes $K_{p,q}(X,B;V)$. On the other hand,  $\mathcal{K}_\bullet$ is obtained by splicing together twists of the exact sequences 
\begin{equation} \label{Wedge.of.M} 
0 \lra \Lambda^{p+1} M_V \lra \Lambda^{p+1} V_X \lra \Lambda^{p} M_V \otimes \OO_X(L) \lra 0. 
\end{equation}
resulting from  \eqref{Seq.Def.M_V}.
The assertion then follows by chasing through the resulting diagram. See \cite[\S1]{SAD} or  \cite[Theorem 5.8]{Eisenbud}  for more details. \end{proof}

\begin{corollary} \label{Higher.Cohom.Interpr.Kpq}
Assume that the vanishings \eqref{Van.Hypoth.Kosz.from.M} are satisfied. Then for $q \ge 2$:
\begin{equation} K_{p,q}(X,B;V) \ = \ \HH{q-1}{X}{\Lambda^{p+q-1} M_V \otimes \OO_X(B + L)}. \end{equation}
 In particular \[
 K_{p,n+1}(X,B;V) \ = \  \HH{n}{X}{\Lambda^{p+n} M_V \otimes \OO_X(B + L)},
 \] and
   $K_{p,q}(X,B;V) = 0$ if $q > n+1$. 
\end{corollary}
\begin{proof} In fact, one finds from  \eqref{Wedge.of.M} that
\begin{align*}
\HH{1}{X}{\Lambda^{p+1} M_V \otimes \OO_X(B + (q-1)L)} \ &= \ \HH{2}{X}{\Lambda^{p+2} M_V \otimes \OO_X(B + (q-2)L)}\\ & = \ \ldots \\ &= \ \HH{q-1}{X}{\Lambda^{p+q-1} M_V \otimes \OO_X(B +L)}.
\end{align*} 
 If $q -1 > n = \dim X$, then of course this $H^{q-1}$ vanishes, yielding the last assertion. \end{proof}

\begin{remark}[$K_{p,0}$] \label{Kp0. Lemma} Assume that 
\begin{equation} \label{Low.Twists.Vanish}
\HH{0}{X}{\OO_X(B - jL)} \ = \ 0 \ \ \text{for} \  \ j > 0. \end{equation}  
Then  
\begin{equation}  \label{Kp0. Statement}
K_{p,0} (X,B:V) \ = \ \HH{0}{X}{\Lambda^pM_V \otimes \OO_X(B)}, 
\end{equation}
and $K_{p, q}(X, B;V) = 0 $ for $q < 0$. Indeed,  the right-hand side of \eqref{Kp0. Statement} computes quite generally the space $Z_{p,0}(X, B;V)$ of Koszul cycles,
and the hypothesis implies that there are no boundaries in that degree. 
\end{remark}

Finally, we recall a useful duality theorem stated by Green \cite[Theorem 2.c.1]{Kosz1}.
\begin{proposition} \label{Duality}
Assume that $X$ is smooth, that $B$ and $L$ satisfy \eqref{Van.Hypoth.Kosz.from.M} and \eqref{Low.Twists.Vanish}, and that in addition
\[
\HH{i}{X}{\OO_X(B + mL) } \ = \ 0 \ \ \text{for}\ \ 0 \, < \, i \, < \, n \  \text{and all }     \ m   \in  \ZZ.  \tag{*} \] Then for $0 \le q \le n+1$ one has isomorphisms:
\begin{align*}
K_{p,q}\big(X,B;V\big)\ &= \ K_{r-p-n,n+1-q}\big(X, B^\pr; V\big)^* \\ &= \ K_{r-p-n,n-q}\big(X, B^\pr + L; V\big)^*,
\end{align*}
where $B^\pr = K_X - B$ and $r = v - 1$. 
\end{proposition}
\noi Note that the hypotheses imply that the module $R(B;L)$ is Cohen-Maculay. For specific values of the parameters, one can get away with fewer vanishings: see  \cite[loc cit]{Kosz1}

\begin{proof}[Sketch of Proof of Proposition \ref{Duality}] 
Assume to begin with that $1 \le q \le n$. By Proposition \ref{Non.Van.Wedge.p+1.Gives.Koszul},   and an argument as in the  proof of Proposition \ref{Higher.Cohom.Interpr.Kpq} using (*), one finds to begin with:
\[ 
K_{p,q}\big(X,B;V\big) \ = \ H^1 \big(\Lambda^{p+1}M_V  (B + (q-1)L )\big) \ = \ H^{n-1} \big(\Lambda^{p+n-1}M_V  (B + (q+1-n)L )\big).
\]
Observe that  the group on the right is Serre dual to 
\[
H^{1}\big(\Lambda^{p+n-1}M_V^*\otimes \OO_X(B^\pr + (n-q-1)L  )\big).
\]
But $M_V$ has rank $r$ and $\det M_V = \OO_X(-L)$, and therefore
\[  \Lambda^{p+n -1} M_V^* \ = \ \Lambda^{r+1-p-n} M_V \otimes \OO_X(L). \]
Putting this together, it results that $K_{p,q}(X,B;V)$ is dual to
\[
\HH{1}{X}{\Lambda^{r+1-p-n}M_V \otimes \OO_X( B^\pr + (n-q)L)}.
\]
The first isomorphism in the statement is then a consequence of Proposition \ref{Non.Van.Wedge.p+1.Gives.Koszul}, while the second follows from the definition of Koszul cohomolgy. The cases $q = 0, n+1$ are proved similarly, using Remark \ref{Kp0. Lemma}. 
\end{proof}

%
%
%

\subsection*{Secant Sheaves} This subsection and the next contain the technical heart of the paper.  We develop  a method for establishing the non-vanishing of certain Koszul cohomology groups via  the presence of suitable secant planes, and prove a theorem that allows one to apply the criterion inductively on dimension. 

Keeping $X \subseteq \PP(V)$ and $L$ as above, fix a non-trivial quotient
\[ \pi :V \twoheadrightarrow W,\] with $\dim W = w < v = \dim V.$ This defines a linear subspace $\PP(W) \subseteq \PP(V)$, and we denote by $Z \subseteq X$ the scheme-theoretic intersection
\[     Z \  =_{\text{def}} \  X \, \cap\, \PP(W)  \]
inside $\PP(V)$. Write 
$J =\ker(\pi)  \subseteq   V$.
Then $Z$ is equivalently the subscheme characterized by the property that  $J$ generates  the twisted ideal sheaf $  \II_{Z/X}( L)$ via the evaluation map $V_X \lra \OO_X(L)$. Note that we do not assume that $Z$ spans $\PP(W)$, i.e. we allow the possibility that the natural map 
\begin{equation} \label{eval.on.Z}
 W \lra \HH{0}{Z}{\OO_Z(L) }
\end{equation}
 has a non-trivial kernel (as well as a non-trivial cokernel). We assume in what follows that $Z \ne \varnothing$.
 
 The map \eqref{eval.on.Z} gives rise to a surjective homomorphism 
 \[ 
 \text{ev}_W : W_X \lra   \OO_Z(L) \]
 of sheaves on $X$. We put 
 \[ \Sigma_W \ = \  \ker ( \text{ev}_W). \]
Thus $\Sigma_W$ is a torsion-free sheaf of rank $w = \dim W$ on $X$, and we have an exact commutative diagram
\begin{equation}\label{M.Sigma.Diagram}
\begin{gathered}
\xymatrix{
0 \ar[r]  & M_V \ar[r] \ar[d]_\rho & V_X \ar[r]^<<<<<{\ev_V} \ar[d]_\pi & \OO_X(L) \ar[r] \ar[d] & 0\\
0 \ar[r] & \Sigma_W \ar[r]^{\iota}& W_X \ar[r]^<<<<{\ev_W} &    \OO_Z(L) \ar[r]&0
}
\end{gathered}
\end{equation}
  with surjective columns.
  
  When $X$ is a smooth curve, $\Sigma_W$  was introduced in \cite{GL1}: in this case it is locally free. In general it is not a vector bundle, but one can give a local description.  \begin{lemma}
  Fix a point $x \in X$. Then  there are \tn{(}non-canonical\tn{)} isomorphisms of germs 
\begin{equation}
\begin{aligned} (\Sigma_W)_x \ & \cong \  (\II_{Z/X})_x & \oplus  \ \big( \OO_{X,x})^{\oplus w-1} \\  (W_X )_x \ & \cong \  \ \  \OO_{X,x} &\oplus \ \big( \OO_{X,x})^{\oplus w-1} \end{aligned}
\end{equation}
 under which the embedding $(\Sigma_W)_x \subseteq (W_X)_x$  goes over to the  evident component-wise inclusion of the two modules on the right. 
  \end{lemma}
  \begin{proof}
 We may suppose that $x \in Z$, and we fix an identification $ \OO_Z(L)_x \cong \OO_{Z,x}$. Then thanks to the surjectivity of $\ev_W$,  there exists a basis 
  \[   e_1, \ldots, e_w\ \in \ \OO_{X,x}^{w} \cong W \otimes_k \OO_{X, x}\] with the property that
 $\ev_W(e_1) = \bar{1} \in \OO_{Z,x}$
  is the canonical generator, while $\ev_W(e_j) = 0$ for $j \ge 2$. The assertion follows. 
  \end{proof}
  
 \begin{corollary}
 There is a canonical surjection $\eps: \Lambda^w \Sigma_W \lra \II_{Z/X}$ sitting in  a commutative diagram
 \begin{equation} \label{Wedge.Sigma.Diagram}
 \begin{gathered}
 \xymatrix{
 \Lambda ^w \Sigma_W \ar[d]_\eps  \ar[r]^{\Lambda^w \iota} & \Lambda^w W_X \ar[d]^{\cong} \\
 \II_{Z/X} \ar[r] & \OO_X ,}  \end{gathered}
 \end{equation}
 whose bottom row is  the natural inclusion of $\II_{Z/X}$ into $\OO_X$. \end{corollary}
\begin{proof}
In fact, it follows from the previous lemma that  $\Lambda^w \iota$ 
factors through the inclusion of $\II_{Z/X}$ into $\OO_X \cong \Lambda^w W_X$.
\end{proof}

Composing $\eps : \Lambda^w \Sigma_W \lra \II_{Z/X}$ with the  canonical map $\Lambda^w \rho : \Lambda^w M_V \lra \Lambda^w \Sigma_W$,  one arrives at a surjective homomorphism
\[   \sigma : \Lambda^w M_V \lra \II_{Z/X}. \]
Our basic strategy for proving the non-vanishing of Koszul groups will be to show that  a twist of $\sigma$ determines a non-zero map on cohomology. This motivates the following
\begin{definition} \label{Def.Plane.Carries.Syzygies}
Fix a divisor $B$ as in the previous subsection, and an integer $q \ge 2$. We say that $W$ \textit{carries weight $q$ syzygies  of $B$} if the mapping
\begin{equation} \label{Map.for.Carrying.Syzygies}
\HH{q-1}{X}{\Lambda^wM_V \otimes \OO_X(B+L)} \lra \HH{q-1}{X}{\II_{Z/X}(B+L)} 
\end{equation}
determined by $\sigma$ is surjective. (When $V = \HH{0}{X}{L}$ and  it is important to emphasize the role of $L$, we will speak of syzygies with respect to $L$.)  \qed
\end{definition}
\noi Assuming that $L$ satisfies the vanishings  in equation \eqref{Van.Hypoth.Kosz.from.M}, we may interpret \eqref{Map.for.Carrying.Syzygies} as a homomorphism
\[
\alpha_{W} : K_{w+1-q,q}(X,B;L) \lra \HH{q-1}{X}{\II_{Z/X}(B+L)} .
\]
If the group on the right vanishes then of course the condition in Definition \ref{Def.Plane.Carries.Syzygies} doesn't yield any information, but in practice we will always arrange things so that \[ \HH{q-1}{X}{\II_{Z/X}(B+L)}\ \ne \ 0.\] In this case, the surjectivity of  \eqref{Map.for.Carrying.Syzygies} guarantees that 
$
K_{w+1-q,q}(X,B;L)  \ne  0$.

\vskip 20pt
\begin{remark}
Observe that if $X$ is a curve, then the condition of the Definition holds automatically (in the one non-trivial case $q = 2$). The main goal of the next few pages is  to develop an inductive criterion for verifying the surjectivity of \eqref{Map.for.Carrying.Syzygies}
 when $X$ has dimension $\ge 2$. \qed
\end{remark}

Assuming henceforth that $n = \dim X \ge 2$, we now analyze the deportment of these constructions upon passing to hyperplane sections. Let $H$ be a very ample divisor on $X$, and let $\ol{X} \in \linser{H}$ be a general divisor in the corresponding linear system.  We may -- and do -- assume that $\ol{X}$ is itself an irreducible variety, and that \eqref{M.Sigma.Diagram}
 remains exact after tensoring by $\OO_{\ol{X}}$. 
 
Fixing notation, let
 \[   V^\pr \ = \ V \, \cap \, \HH{0}{X}{\II_{\ol{X}/X}(L)}, \]
the intersection taking place inside $\HH{0}{X}{\OO_X(L)}$. Put $W^\pr = \pi(V^\pr) \subseteq W$, and write
\begin{equation} \label{Restrict.VW} \overline{V} = V / V^\pr  \ \ , \ \ \overline{W} = W / W^\pr. \end{equation}
We thus get a commutative diagram of exact sequences
\begin{equation} \label{Diagr.Vect.Spaces}
\begin{gathered}
\xymatrix{
0 \ar[r] & V^\pr \ar[r] \ar[d]_{\pi^\pr} &V \ar[r] \ar[d]_\pi&\overline{V} \ar[r] \ar[d]_{\overline{\pi}} &0\\
0 \ar[r] & W^\pr \ar[r] &W \ar[r] &\overline{W} \ar[r] & 0  
}
\end{gathered}
\end{equation}
of vector spaces. We denote  by $v^\pr, \ol{v}, w^\pr, \ol{w}$ the dimensions of the spaces in question, so that
\[  v \ = \ v^\pr \, + \, \ol{v} \ \  , \ \ w \ = \ w^\pr \, + \, \ol{w}. \]

On the other hand, put 
\[  \overline{L}  \ = \  L \mid \ol{X} \ \ , \ \ \ol{Z} \ = \  Z \cap \ol{X}. \]
Then the previous constructions apply on $\ol{X}$, and we get an exact commutative diagram 
\begin{equation}\label{M.Sigma.Restricted.Diagram}
\begin{gathered}
\xymatrix{
0 \ar[r] & M_{\ol {V}} \ar[r] \ar[d] \ar[d]_{\ol{\rho}}& \ol{V}_{\ol{X}} \ar[r]^<<<<<<{\ev_{\ol{V}}} \ar[d]_{\ol{\pi}}   & \OO_{\ol{X}}(\ol{L}) \ar[r] \ar[d] & 0\\
0 \ar[r] & \Sigma_{\ol{W}} \ar[r] & \ol{W}_{\ol{X}} \ar[r]^<<<<<{\ev_{\ol{W}}}   &   \OO_{\ol{Z}}(\ol{L} )\ar[r]&0
}
\end{gathered}
\end{equation}
of sheaves on $\ol{X}$. This defines the bundle $M_{\ol {V}}$ and the sheaf $\Sigma_{\ol{W}}$. We also have a surjective homomorphism
$
\ol{\sigma} :  \Lambda^{\ol{w}} M_{\ol{V}} \lra \II_{\ol{Z}/\ol{X}}. 
$
Finally, write $\ol{B}, \ol{H}$ for the restrictions of $B$ and $H$ to $\ol{X}$. Note that if $\dim Z = 0$, then $\ol{Z} = \varnothing$, and $\Sigma_{\ol{W}} = \ol{W}_{\ol{X}}$. In this case, if moreover $W^\pr = W$, then $\ol{w} = 0$ and  by convention  we take $\ol{\sigma} $ to be the identity $ \OO_{\ol{X}} \lra \OO_{\ol{X}}$.

Our first main technical result gives an inductive criterion for verifying the surjectivity appearing in Definition \ref{Def.Plane.Carries.Syzygies}.
\begin{theorem} \label{First.Technical.Thm}
Fix $q \ge 2$, and suppose that
\begin{equation} \label{Van.of.Hi}
\HH{q-1}{X}{\II_{Z/X}(B + L + H)} \ = \ 0. \end{equation}
Furthermore,  assume that $\ol{W}$ carries weight $(q-1)$ syzygies of $\ol{B}+ \ol{H}$   on $\ol{X}$. Then $W$ carries weight $q$ syzygies of $B$   on $X$. 
\end{theorem}
\noi When $q = 2$, the hypothesis on $\ol{X}$ is that the mapping
\[
\HHH{0}{\ol{X}}{\Lambda^{\ol{w}}M_{\ol{V}} \otimes \OO_{\ol{X}}(\ol{B} + \ol{H} +\ol{L})}  \lra \HHH{0}{\ol{X}}{\II_{\ol{Z}/\ol{X}}(\ol{B}+\ol{H}+ \ol{L})} 
\]
determined by $\ol{\sigma}$ be surjective. If $\ol{w} = 0$ this condition is satisfied automatically. 

The second result asserts that when the hypotheses of Theorem \ref{First.Technical.Thm}
are satisfied, one gets (under mild additional assumptions) non-vanishing classes in $K_{p,q}$ for many different $p$. In the following statement, we keep the notation introduced above for the various dimensions arising in the construction.
\begin{theorem} \label{Second.Technical.Theorem}
In addition to the hypotheses of Theorem  \ref{First.Technical.Thm}, assume that 
$Z \subseteq X$ is a local complete intersection, that $\ol{v}- \ol{w} > n$, and that
\begin{equation} \label{Non.Van.Hypoth.Thm.1.9}
\HH{q-1}{X}{\II_{Z/X}(B+L)} \ \ne \ 0.
\end{equation}
Assume also that the vanishings \eqref{Van.Hypoth.Kosz.from.M} hold. 
Then $K_{p,q}(X, B; V) \ne 0 $ for all 
\[ w + 1 - q \ \le \ p \ \le \ (v^\pr + \ol{w}) + 1 - q. \] 
\end{theorem}
\noi 
The proofs of both theorems appear in the next subsection.
Notice that neither statement gives any information about the groups $K_{p,1}$; these will be analyzed separately in \S \ref{Kp1.Section}. 

\subsection*{Proofs of Theorems \ref{First.Technical.Thm}
 and \ref{Second.Technical.Theorem}
}

We start by studying the restriction of the diagram \eqref{M.Sigma.Diagram}
 to the general divisor $\ol{X} \in \linser{H}$. 

\begin{lemma} \label{Splitting.Proposition}
One can choose isomorphisms
\begin{equation} \label{M.Sigma.Splittings}
\begin{aligned} M_V \otimes \OO_{\ol{X}} \ &\overset{\cong}{\lra} \ V^\pr _{\ol{X} }\, \oplus \, M_{\ol{V}} \\ \Sigma_W \otimes \OO_{\ol{X}} \  &\overset{\cong}{\lra} \ W^\pr _{\ol{X}}\, \oplus \, \Sigma_{\ol{W}} 
\end{aligned}
\end{equation}
under which the quotient $\rho \otimes 1: M_V \otimes \OO_{\ol{X}} \lra \Sigma_{W} \otimes \OO_{\ol{X}} $ is identified with the direct sum of the maps
\[
\pi^\pr  : V^\pr_{\ol{X}} \lra W^\pr _{\ol{X}}  \ \ \ , \ \ \  \ol{\rho} : M_{\ol{V}} \lra \Sigma_{\ol{W}}. 
\]
\end{lemma}

\begin{proof}  
Note to begin with that by construction $V^\pr$ is the kernel of 
\[  V \ = \ \HH{0}{{\ol{X}}}{V_{\ol{X}}} \lra \HH{0}{{{\ol{X}}}}{\OO_{\ol{X}}(\ol{L})}, \]
and similarly the sections in $W^\pr$ vanish on $
\OO_{\ol{Z}}(\ol{L}) $. 
We then get an exact commutative diagram
\begin{equation}\label{M.Sigma.onH.Diagram}
\begin{gathered}
\xymatrix{
0 \ar[r]  & V^\pr_{\ol{X}} \ar[r] \ar[d]_{\pi^\pr} & M_V \otimes \OO_{\ol{X}} \ar[r] \ar[d]_{\rho \otimes 1} & M_{\ol{V}} \ar[r] \ar[d]_{\ol{\rho}} & 0\\
0 \ar[r]  & W^\pr_{\ol{X}}\ar[r]   & \Sigma_W \otimes \OO_{\ol{X}} \ar[r]   & \Sigma_{\ol{W}} \ar[r] & 0
}
\end{gathered}
\end{equation}
of sheaves on ${\ol{X}}$. It remains to show that one can find compatible splittings of the two rows. 

To this end, referring back to \eqref{Diagr.Vect.Spaces}, choose a section $ \ol{V} \lra V$ of the quotient $V \lra \ol{V}$ that maps $\ker(\ol{\pi})$ into $\ker(\pi)$. This determines a compatible splitting of the two rows of \eqref{Diagr.Vect.Spaces}. The left-hand square of \eqref{M.Sigma.Diagram}
 then restricts to
\begin{equation}  \label{LH.Square.Spitting}
\begin{gathered} \xymatrix{
  M_V \otimes \OO_{\ol{X}} \ar[r] \ar[d]_{\rho \otimes 1} & V^\pr_{\ol{X}} \oplus \ol{V}_{\ol{X}} \ar[d]_{\pi^\pr \oplus \ol{\pi}} \\ \Sigma_W \otimes \OO_{\ol{X}} \ar[r]  & W^\pr_{\ol{X}} \oplus \ol{W}_{\ol{X}} } 
  \end{gathered} \notag
 \end{equation}
and this in turn yields the required splitting of \eqref{M.Sigma.onH.Diagram}.
\end{proof}

\begin{proof}[Proof of Theorem  \ref{First.Technical.Thm}]
Starting from the diagram
 \[ \xymatrix{ 
0 \ar[r] &\Lambda^w M_V (B+L) \ar[r]^<<<<<{\cdot {\ol{X}}}  \ar[d]_{\sigma} &\Lambda^w M_V  (B+L+H) \ar[r]  \ar[d]_{\sigma}  &\Lambda^w M_V\otimes \OO_{\ol{X}} (\ol{B} + \ol{L} + \ol{H}) \ar[r] \ar[d]_{\ol{\sigma}}  &0 \\
0 \ar[r] & \II_{{Z}/X}({B} + {L}) \ar[r]^<<<<<<{\cdot {\ol{X}}} & \II_{Z/X}(B+L+H) \ar[r] & \II_{\ol{Z}/{\ol{X}}}\otimes \OO_{\ol{X}}(\ol{B} + \ol{L} + \ol{H}) \ar[r]&0
}\]
we obtain an exact commutative diagram:
\begin{equation} \label{Diagram.on.Cohom}
\Small
\begin{gathered}
\xymatrix{
  \Hh{q-2} { \Lambda^{{w}} M_{{V}}\otimes \OO_{\ol{X}}  (\ol{B} + \ol{L} + \ol{H})} \ar[r]   \ar[d]& \Hh{q-1} {\Lambda^w M_V  (B+L)} \ar[r] \ar[d] & \Hh{q-1}{ \Lambda^w M_V (B+L+H)} \ar[d]
  \\
  \Hh{q-2} {\II_{{Z/X}} \otimes \OO_{\ol{X}}  (\ol{B} + \ol{L} + \ol{H})} \ar[r] & \Hh{q-1}{\II_{Z/X}(B+L)} \ar[r] & \Hh{q-1}{\II_{Z/X}(B+L+H)}
  }
\end{gathered}
\end{equation}
\normalsize
whose middle column is \eqref{Map.for.Carrying.Syzygies}. In view of the vanishing \eqref{Van.of.Hi}, it suffices to show that the  vertical homomorphism on the left -- which is induced by the restriction of $\sigma$ to $\ol{X}$ -- is surjective. 

To this end,  consider the restriction to ${\ol{X}}$
\[   \Lambda^w  (M_V \otimes \OO_{\ol{X}})  \lra \Lambda^w( \Sigma_W \otimes \OO_{\ol{X}}) \lra \II_{Z/X} \otimes \OO_{\ol{X}} \]
of the composition defining $\sigma$. 
 Using the identifications of Lemma \ref{Splitting.Proposition}, we see that it occurs as the bottom row of the commutative diagram
\begin{equation} \label{big.tensor.prod.diagram}
\begin{gathered}
\xymatrix{
   \Lambda^{w^\pr}V^\pr_{\ol{X}} \otimes \Lambda^{\ol{w}} M_{\ol{V}} \ar[d] \ar[rr]^{\Lambda^{w^\pr}  \pi^\pr  \otimes \Lambda^{\ol{w}}\ol{\rho}} & & \Lambda^{w^\pr}W^\pr_{\ol{X}} \otimes \Lambda^{\ol{w}} \Sigma_{\ol{W}}   \ar[r]^{1 \otimes \ol{\eps}}  \ar[d] &  \Lambda^{w^\pr}W^\pr_{\ol{X}} \otimes \II_{\ol{Z}/{\ol{X}}}  \ar[d]^{\cong}\\
    \Lambda^w \big( V^\pr_{\ol{X}}  \oplus M_{\ol{V}} \big) \ar[rr] & &
 \Lambda^w \big( W^\pr_{\ol{X}}  \oplus \Sigma_{\ol{W}} \big) \ar[r]^{\eps|H} & \II_{Z/X} \otimes \OO_{\ol{X}}, 
 }
\end{gathered}
\end{equation}
where the first two vertical maps  are those arising from the decomposition of the exterior product of a direct sum.   So for the surjectivity of the left-hand map in \eqref{Diagram.on.Cohom}, it is sufficient to establish the surjectivity of the homomorphism 
\[
H^{q-2}\Big({\ol{X}} , { \Lambda^{w^\pr}V^\pr_{\ol{X}} \otimes \Lambda^{\ol{w}} M_{\ol{V}}        (\ol{B} + \ol{L} + \ol{H})}\Big) \lra H^{q-2}\Big({\ol{X}},  \Lambda^{w^\pr} W^\pr_{\ol{X}} \otimes \II_{\ol{Z}/H}(\ol{B} + \ol{L} + \ol{H}) \Big)
\]
\normalsize
coming from the top row of \eqref{big.tensor.prod.diagram}. But the map in question is identified with 
\[
\xymatrix{
\Lambda^{w^\pr}V^\pr \otimes H^{q-2} \Big({\ol{X}},  \Lambda^{\ol{w}} M_{\ol{V}}   (\ol{B} + \ol{L} + \ol{H}) \Big) \ar[rr]^>>>>>>>>>{\Lambda^{w^\pr}\pi^\pr \otimes \ol{\sigma}} & &\Lambda^{w^\pr} W^\pr \otimes H^{q-2}\Big({\ol{X}},  \II_{\ol{Z}/{\ol{X}}}  (\ol{B} + \ol{L} + \ol{H})\Big ),
}
\]
\normalsize whose surjectivity is implied by the assumption  that $\ol{W}$ carries syzygies of $\ol{B} + \ol{H}$ on ${\ol{X}}$. 
\end{proof}

\begin{remark} \label{Wbar.Eqs.Zero.Remark}
Note   that if $q =2$ and $W^\pr = W$, so that $\ol{w} = 0$, then the argument just completed works with our convention that $\ol{\sigma} = \textnormal{id}: \OO_{\ol{X}} \lra \OO_{\ol{X}}$ (in which case the surjectivity of the map on $H^{0}$ is automatic). 
\end{remark}

For the second Theorem, the plan is to deduce   additional surjectivities by ``enlarging $W$." 
\begin{proof}[Proof of Theorem \ref{Second.Technical.Theorem}] Suppose that $\pi_1 :V \twoheadrightarrow U $ is a quotient factoring $\pi$:
\[  V \overset{\pi_1} \lra U \overset{\pi_2} \lra W.\]
This gives rise to a chain of subspaces 
$ \PP(W) \subseteq \PP(U) \subseteq \PP(V)$.  Set
\[   U^\pr \ = \ \pi_1(V^\pr) \ \ , \ \ \ol{U} \ = \ U/U^\pr, \]
and write $u, u^\pr,  \ol{u}$ for the dimensions of the three vector spaces in question, so that $u = u^\pr + \ol{u}$. Assume now that $U$ satisfies two properties:
\begin{gather} \text{ The natural map } \ol{U} \lra \ol{W} \text{ is an isomorphism}; 
\label{First.Assumption.U}
 \\
\PP(U) \, \cap \, X \ = \ Z. \label{Second.Assumption.U}
\end{gather}
Once such a quotient is at hand, the constructions above -- with $U$ in place of $W$ -- give rise to maps
\[  
\eps^\# : \Lambda^u \Sigma_U \lra \II_{Z/X} \ \ , \ {\sigma}^\# : \Lambda^u M_V \lra \II_{X/ Z}.
\]
Furthermore, thanks to \eqref{First.Assumption.U}, the homomorphism
\[  \ol{\sigma}^\# : \Lambda^{\ol{u}} M_{\ol{V}} \lra \II_{Z/X} \otimes \OO_{\ol{X}}\]
determined by $  \ol{U}$ coincides with the homomorphism $\ol{\sigma} : \Lambda^{\ol{w}} M_{\ol{V}} \lra \II_{\ol{Z}/{\ol{X}}}$ coming from $\ol{W}$. Therefore Theorem \ref{First.Technical.Thm} implies that $U$ carries weight $q$ syzygies of $B$, and it then follows from \eqref{Non.Van.Hypoth.Thm.1.9}
 that
\begin{equation} \label{Which.Kpq.Non.Zero} K_{u + 1 - q,q}(X, B;L) \ \ne \ 0. \end{equation}
So it remains only to construct $U$ with the  appropriate range of dimensions. 

To this end, we consider a linear subspace $\Lambda \subseteq V$, and ask when $U =_{\text{def}} V/\Lambda$ meets the stated conditions. In order to get a factorization of $\pi$, it is necessary first of all that 
\[
\Lambda \ \subseteq \ \ker(\pi). 
\]
Once this holds, \eqref{First.Assumption.U} is equivalent to asking that 
\begin{equation}
V^\pr \, + \, \Lambda \ = \ V^\pr \, + \, \ker(\pi) \tag{*}
\end{equation}
(since the two sides are the kernels of the maps from $V$ to $\ol{U}$ and $\ol{W}$ respectively), and \eqref{Second.Assumption.U} is equivalent to the condition
\begin{equation}
\text{ The sections in  } \Lambda \ \text{generate  the sheaf } \II_{Z/X}\otimes \OO_X(L) . 
\tag{**} 
\end{equation} Now \[ \dim \big( V^\pr + \ker(\pi) \big) \ = \ v - \ol{w}, \]
so one sees that (*) will hold if $\Lambda \subseteq \ker(\pi)$ is a general subspace with 
\[ \dim \Lambda \ \ge \ \ol{v} - \ol{w}. \]
As for (**), recall  that we assume that $Z$ is a local complete intersection. Since by assumption $\ker(\pi)$ generates $\II_{Z/X}$, it is elementary (and well-known) that so too will any general subspace of $\ker(\pi)$ having $\dim \ge n+1$. Therefore, as soon as $\ol{v} - \ol{w} > n$, we can find a quotient $\pi_1 : V \lra U$ having the required properties with $u = \dim U$ taking any value in the interval 
\[ w \ \le \ u \ \le \ v^\pr + \ol{w}. \]
Combining this with \eqref{Which.Kpq.Non.Zero}, we get the assertion of the Theorem. 
\end{proof}

  %
  %
  %
  
\section{Asymptotic Non-Vanishing Theorem} \label{Asympt.Nonvan.Section}

This section is devoted to the main asymptotic non-vanishing theorem.

We start by setting notation. Let $X$ be a smooth projective variety of dimension $n \ge 2$. We fix an ample divisor $A$ on $X$, and an arbitrary divisor $P$, and put:
\begin{equation} 
\begin{gathered}
L_d \ = \ dA + P, \\ 
V_d \ = \ \HH{0}{X}{\OO_X(L_d)}  \ \ , \ \  v_d \ =_{\text{def}} \dim V_d  .\end{gathered}
\end{equation}
We will always assume in the sequel that $d$ is taken to be sufficiently large so that $L_d$ is very ample. Thus $L_d$ defines an embedding
\[    X\  \subseteq\  \PP(V_d) \ = \ \PP^{r_d}, \]
where as usual $r_d = v_d - 1$. 
Observe that Serre vanishing applies to twisting by $L_d$ for large $d$.    Note also that thanks to Riemann--Roch, $v_d$ and $r_d$ are given for $d \gg 0$ by polynomials of degree $n$ in $d$. 

The aim in this section is to establish:
\begin{theorem} \label{Main.NonVan.Statement.Section.2}
Fix an index $2 \le q \le n$, and let $B$ be a divisor on $X$.
There exist constants $C_1\, , \, C_2\, > \, 0$ with the property that if $d$ is sufficiently large then 
 \[K_{p,q}(X,B; L_d) \ \ne \ 0 \]
 whenever
 \[   C_1 \cdot d^{q-1} \ \le \ p \ \le \ r_d \, - \, C_2 \cdot d^{n-1}. \]
 If moreover $\HH{i}{X}{\OO_X(B)} = 0$ for $0 < i < n$, then  the non-vanishing holds in the range 
 \[   C_1 \cdot d^{q-1} \ \le \ p \ \le \ r_d \, - \, C_2 \cdot d^{n-q}. \]
\end{theorem}
\noi The statement of the Theorem also holds  when $q = 1$ (Proposition \ref{Kp1.Proposition}).

 \begin{remark} \label{Zhou.Remark}
 We note that Zhou \cite{Zhou} has recently been able to adapt the argument below to show that the Theorem remains true even if $X$ is singular. He also relaxes the assumption on $B$. 
 \end{remark}

The plan naturally enough is to prove the theorem inductively on $n$, using Theorems \ref{First.Technical.Thm}
 and \ref{Second.Technical.Theorem}. In the next subsection, we construct and study  the various secant linear spaces and subschemes to which we will apply those results.


\subsection*{Constructions}  It will be convenient to first produce an appropriate subscheme $ Z_d \subseteq X$, and subsequently to  realize $Z_d$ as the intersection of $X$ with a linear subspace $\PP(W_d)$ of $\PP(V_d)$.\footnote{Our choice of $Z$ was inspired by the argument of Eisenbud et.\ al.\ in     \cite[Proof of Proposition 3.2]{EGHP}.}

As above, we suppose that $B$ is a fixed divisor on $X$. We begin by choosing a very ample divisor  $H$ on $X$ with the properties that:
   \begin{gather}
 H + B - K_X \ \text{ is very ample; } 
\label{First.Cond.Eqn} \\
\HH{i}{X}{\OO_X(mH   + B )} \ = \ 0 \ \ \text{for } i > 0 \  , \ m \ge 1;   \label{Second.Cond.Eqn}
\\
 \HH{i}{X}{\OO_X(K_X + mH)} \ = \ 0\ \  \text{for } i > 0 \  , \ m \ge 1.  \label{Third.Cond.Eqn}
\end{gather}
For the purposes of a future induction, we fix at this point smooth irreducible divisors
\[  \ol{X}_1  \, ,\, \ldots \, , \, \ol{X}_n \ \in \ \linser{H} \]
that meet transversely. 
Next, set $ c = n +2 - q$ and assume that $d$ is sufficiently large so that $L_d - (c-1)H$ is very ample. Pick  divisors
\begin{equation} \label{Def.of.Di.Eqn}
\begin{gathered}
D_1 \ , \ D_2 \ , \ \ldots \ , \ D_{c-2} \ \in \ \linser{ H } \\ 
D_{c-1} \ \in \ \linser{H + B - K_X } 
\ \ , \ \ 
D_c \ \in \ \linser{ L_d - (c-1) H } \end{gathered}
\end{equation}
in such a way that $\sum \ol{X}_j + \sum D_i$ has simple normal crossings, and  let 
\begin{equation} \label{Def.of.Z.Eqn}
Z \ = \ Z_d \ = \ D_1 \, \cap \, \ldots \, \cap D_c 
\end{equation}
be the complete intersection of the $D_i$. Thus $Z$ is smooth, with $\dim Z = n -c = q-2$. This will be the secant subvariety with which we work. 
 Our choice of the $D_i$ is arranged so that 
\begin{equation}
D_1 + \ldots + D_c \ \lin \ L_d + B - K_X, \label{Sum.of.Di.Eqn} \end{equation}
and so that only the last divisor $D_c$ involves $L_d$. We always take $d$ to be large enough so that each $\OO_X(L_d - D_i)$ is base-point free. Denoting  by $ \II_{Z_d/X}$ the ideal sheaf of $Z_d$ in $X$, this implies that $\II_{Z_d/X} \otimes \OO_X(L_d)$ is globally generated.

The next step is to define a quotient $V_d \twoheadrightarrow W_d$ whose projectivization $\PP(W_d) \subseteq \PP(V_d)$ meets $X$ along $Z_d$. To this end, let
\begin{equation} \label{Def.J1.Eqn}
J_{0,d} \ = \ \HH{0}{X}{\II_{Z_d/X} \otimes \OO_X(L_d)} ,
\end{equation}
and put
$W_{0,d} =V_d/ J_{0,d}$.
Thus
\[  W_{0,d} \ = \ \Image \Big( \HH{0}{X}{\OO_X(L_d)} \lra \HH{0}{Z_d}{\OO_{Z_d}(L_d)}  \Big ). \]
Moreover,  the global generation of $\II_{Z_d/X}(L_d)$ implies that   $\PP(W_{0,d}) \subseteq \PP(V_d)$ is a subspace whose intersection with $X$ is precisely $Z_d$. In principle, we would like to apply the constructions of the previous section with this subspace. However in order for the induction to run smoothly, it is convenient to work instead with a possibly larger linear space in which the codimension of  $\PP(W_{0,d})$ is bounded independently of $d$. 

Suppose then that for every sufficiently large $d$, one chooses (or is given)  a subspace $J_d \subseteq J_{0,d}$ that satisfies:  \begin{equation}
\label{First.Cond.Wd}
\text{The sections in $J_d$ generate $\II_{Z_d/X} \otimes \OO_X(L_d)$;}
\end{equation}
\begin{equation} \label{Second.Cond.Wd}
\dim \big ( \, J_{0,d} / J_d \, \big) \ \le \ a, 
\end{equation}
where $a$ is some constant independent of $d$. Set
\begin{equation}  \label{Def.Wd.Eqn}
W_d \ = \ V_d / J_d \ \ , \ \ w_d \ =_{\text{def}} \ \dim W_d,  \end{equation}
and in accordance with the notation in \S \ref{Secant.Constructions.Section}, write
\begin{equation} \label{Def.pid}
  \pi_d : V_d \lra W_d
\end{equation}
for the canonical map.
The first condition \eqref{First.Cond.Wd}
guarantees that  $\PP(W_d) \cap X = Z_d$ inside $\PP(V_d)$. 

It will be useful to have some terminology for the various constructions and conditions  just introduced. 
\begin{definition}  \label{Def.Adapted}
We say that the data $H, \ol{X}_j, Z_d$ and $W_d$ are \textit{adapted to $B$} if:
\begin{itemize}
\item[(i).] $H$ satisfies the conditions
 \eqref{First.Cond.Eqn} -- \eqref{Third.Cond.Eqn};
 \vskip 5pt
\item[(ii).] Each $Z_d$ (for large $d$) is constructed as in \eqref{Def.of.Di.Eqn} and \eqref{Def.of.Z.Eqn};
 \vskip 5pt
\item[(iii).] Every $W_d$ (for large $d$) arises from a subspace $J_d $ that satisfies \eqref{First.Cond.Wd} and \eqref{Second.Cond.Wd}. 
\end{itemize} 
 \end{definition}
\noi We will prove in the next subsection that if the data are adapted to $B$, and if $d$ is sufficiently large, then $W_d$ carries weight $q$ syzygies of $B$ with respect to $L_d$ (in the sense of Definition \ref{Def.Plane.Carries.Syzygies}).
 For now we study these constructions a little further.

The first point is to work out the cohomological properties of the ideal sheaf $\II_{Z_d/X}$.
\begin{lemma} \label{Cohomol.Properties.Lemma} \label{Cohom.Properties.Ideal.Lemma}
 Assume that $H$, $ Z_d $, and $ W_d  $ are adapted to $B$. If $d$ is sufficiently large, then:
 \item[(i).] One has 
\begin{equation} \label{Higher.Cohom.Twists.Ideal.Z}
\begin{gathered}
\Hh{q-1}{\II_{Z_d/X} \otimes \OO_X(B + L_d)} \ \ne 0 \\  \Hh{q-1}{\II_{Z_d/X} \otimes \OO_X(B +L_d +H)} \ = \ 0. 
\end{gathered}
\end{equation}
\item[(ii).] For $i > 0$ the dimensions
\begin{equation} \dim \HH{i}{X}{\II_{Z_d/X} \otimes \OO_X(L_d)}  \ \ \text{and} \ \ \dim \HH{i}{X}{\II_{Z_d/X} \otimes \OO_X(L_d -H)}
\end{equation}
are bounded above independently of $d$. 
 \end{lemma}

\begin{proof}
Put 
\[ E \ = \ \bigoplus_{i =1}^c \, \OO_X(-D_i), \]  so that $\II_{Z_d/X}$ is resolved by the Koszul complex whose $j^{\text{th}}$ term is $\Lambda^jE$. We use this resolution to compute the cohomology of $\II_{Z_d/X}$ and its twists.  For the first assertion of (i) it suffices to show that
\begin{gather*}
\HH{n}{X}{\Lambda^c E\otimes \OO_X(B + L_d)} \ \ne \ 0,  \tag{*} \\
\HH{i}{X}{\Lambda^jE \otimes \OO_X(B + L_d)} \ = \ 0 \ \text{ for } i > 0 \ , \ j < c . \tag{**}
\end{gather*}
Thanks to \eqref{Sum.of.Di.Eqn} one has $\Lambda^c E \otimes \OO_X(B +L_d)= \OO_X(K_X)$, which implies (*). As for (**), note that if $j < c$, then $\Lambda^jE$ is a direct sum of twists of $\OO_X(K_X)$ by line bundles of the form $\OO_X(mH)$ for $m \ge 1$, as well as possibly one or both of
\[
\OO_X( H + B - K_X) \ \ , \ \ \OO_X(L_d - (c-1)H).
\]
Summands involving $L_d$ can be made to have vanishing cohomology by taking $d \gg 0$, while the remaining terms are covered by \eqref{Second.Cond.Eqn} and \eqref{Third.Cond.Eqn}. This proves the first assertion of (i), and the second is similar.  Turning to  (ii), observe that the line bundle summands appearing in $\Lambda^j E \otimes\OO_X( L_d)$ and $\Lambda^j E \otimes \OO_X(L_d-H)$ either involve only $K_X$, $B$ and $H$, or else can be made to have vanishing higher cohomology by taking $d$ large. Thus for $i > 0$ the dimension of each of the groups
\begin{equation}
\HH{i}{X}{\Lambda^j E \otimes \OO_X(L_d)} \ \ \text{and} \ \ \HH{i}{X}{\Lambda^j E \otimes\OO_X( L_d-H)} \
\notag \end{equation}
 is  independent of $d$ (when $d \gg 0$). Statement (ii) then follows from the spectral sequence relating the cohomology of $\II_{Z_d/X} \otimes \OO_X(L_d)$ to the cohomology of the terms of a resolution.  \end{proof}

 In order to apply Theorem \ref{Second.Technical.Theorem}, it will be important to estimate $w_d = \dim W_d$. The next statement shows that the dimension in question grows like $d^{q-1}$. 
\begin{lemma} \label{Estimate.wd.Lemma}
There is a polynomial $Q(d)$ of degree $q-1$ in $d$ such that the difference
\[   \mid \dim W_{d}  - Q(d) \mid  \] is bounded.
\end{lemma}
\begin{proof}
Thanks to \eqref{Second.Cond.Wd}, it is enough to produce $Q(d)$ such that $ |\dim W_{0,d} - Q(d)|$ is bounded. To this end, let 
\[
Y  \ = \ D_1 \cap \ldots \cap D_{c-1}. 
\]
This is a smooth variety of dimension $n - c + 1= q-1$, and $Z_d$ is realized in $Y$ as the divisor of a section of $\OO_Y\big(L_d - (c-1)H\big)$. The exact sequence
\[
0 \lra \OO_{Y}((c-1)H) \lra \OO_Y(L_d) \lra \OO_{Z_d}(L_d) \lra 0 
\]
shows that $\hh{0}{Z_d}{\OO_{Z_d}(L_d)}$  and $\hh{0}{Y}{\OO_Y(L_d)}$ have bounded difference as functions of $d$, while thanks to Riemann-Roch $\hh{0}{Y}{\OO_Y(L_d)}$ is given for $d \gg 0$ by a polynomial $Q(d)$ of degree $  q-1= \dim Y$. Finally, note that 
\[ \dim \HH{0}{Z_d}{\OO_{Z_d}(L_d )} \ - \ \dim W_{0,d}   \ =\ \dim \HH{1}{X}{\II_{Z_d/X} \otimes \OO_X(L_d) } \]
provided that $d$ is sufficiently large so that $\HH{1}{X}{\OO_X(L_d)} = 0$. But   Lemma \ref{Cohom.Properties.Ideal.Lemma} (ii) implies that the term on the right is bounded, and the Lemma follows. 
\end{proof}

Finally, 
 let
\begin{equation}
M_{L_d} \ = \ \ker \Big ( \HH{0}{X}{\OO_X(L_d)} \otimes_{\bf{k}} \OO_X \lra \OO_X(L_d) \Big) 
\end{equation}
be the vector bundle \eqref{Seq.Def.M_V} arising as the kernel of the evaluation map. The following   result will be used to start the induction in the next subsection.\footnote{Strictly speaking, we will apply this on a hyperplane section of $X$ rather than on $X$ itself, but it seems cleanest to present the Proposition here.}

\begin{proposition}  \label{Surjectivity.H0.Lemma} Suppose we are given for every sufficiently large $d$ a trivial quotient 
\[   s_d:  M_{L_d} \lra T_d \otimes_k \OO_X \lra 0\]
of $M_d$, where $T_d$  is a vector space of dimension 
\[ t_d \ =\  \dim T_d \ \le \ c\] bounded above independent of $d$.  Then for any fixed divisor $C$, the homomorphism
\small
\[
\HHH{0}{X}{\Lambda^{t_d} M_{L_d}  \otimes \OO_X(L_d + C)} \lra \HHH{0}{X}{\Lambda^{t_d} T_d \otimes \OO_X(L_d + C)}  \, = \, \HH{0}{X}{\OO_X(L_d + C)}
\]
\normalsize
determined by $\Lambda^{t_d} s_d$ is surjective for all $d \gg 0$. \end{proposition}

With $H$ the very ample divisor chosen above, the first point is:
 \begin{lemma} 
 If $d$ is sufficiently large, then the vector bundle $M_{L_d}\otimes \OO_X(H)$ is globally generated. In particular, $M_{L_d}$ is a quotient of a direct sum $\oplus \, \OO_X(-H)$ of copies of $\OO_X(-H)$.
 \end{lemma}
\begin{proof} 
Note to begin with that there is a natural map of vector bundles:
\[
\HH{0}{X}{\OO_X(L_d - H)} \otimes_{{k}} M_H \lra M_{L_d}. \tag{*} \]
In fact, the fibres of $M_H$ and  $M_{L_d}$   at a point $x \in X$ are canonically identified with
\[   \HH{0}{X}{\OO_X(H) \otimes \II_{x}} \ \ \text{and} \ \  \HH{0}{X}{\OO_X(L_d) \otimes \II_{x}}\]
respectively, and then (*) is just the globalization of the multiplication map:
\[
\HH{0}{X}{\OO_X(L_d - H)}  \otimes \HH{0}{X}{\OO_X(H) \otimes \II_{x}} \lra \HH{0}{X}{\OO_X(L_d) \otimes \II_{x}} \tag{**}.\footnote{We leave it as an exercise for the interested reader to give a direct global construction of (*).}
\] On the other hand, since $H$ is very ample, $\OO_X(H) \otimes \II_x$ is globally generated for all $x \in X$, and it follows that the map (**) is surjective for all $x \in X$ when  $d \gg 0$. Thus $M_d$ is a quotient of a direct sum of copies of $M_H$ provided that $d$ is sufficiently large. But $M_H \otimes \OO_X(H)$ is itself globally generated, which completes the proof. 
\end{proof}

\begin{proof}[Proof of Proposition \ref{Surjectivity.H0.Lemma}]
By the previous Lemma,  there exists a surjection
\[
 \OO^N_X(-H) \lra M_{L_d} \lra 0\]
 for suitable $N = N_d$,
and therefore also a map
\[  u_d : \OO^N_X(-H) \lra T_d \otimes_k \OO_X \lra 0.\]
It suffices to show that the resulting homomorphism
\begin{equation}
\HHH{0}{X}{\Lambda^{t_d}\big( \OO_X^N(-H ) ) \otimes \OO_X(L_d + C)} \lra \HHH{0}{X}{\OO_X(L_d + C)}
\tag{*} \end{equation}
determined by $\phi_d =_{\text{def}} \Lambda^{t_d} u_d$ is surjective on global sections if $d \gg 0$. Now $\phi_d$ is resolved by an exact Eagon-Northcott complex, which takes the shape
\small
\[   \ldots \lra \OO_X(-(t_d +2)H) \otimes_k S_2 \lra \OO_X(-(t_d +1)H) \otimes_k S_1 \lra \Lambda^{t_d}( \OO_X^N)(-t_dH)\overset{\phi_d}\lra \OO_X \lra 0\]
\normalsize
where the $S_i$ are certain vector spaces (with dimensions depending on $t_d$). It is in turn sufficient for (*) to know that
\begin{align*} \HH{1}{X}{\OO_X(L_d+C - (t_d + 1)H)}\  &=  \ \HH{2}{X}{\OO_X(L_d + C - (t_d + 2)H)}   \\ & \ldots \\ &= \  \HH{n}{X}{\OO_X(L_d +C - (t_d + n)H)} \ = \ 0.
\end{align*}
But since $t_d$ is bounded independent of $d$, we can arrange for this by taking $d \gg 0$. 
\end{proof}


\subsection*{Proof of Non-Vanishing Theorem} The plan   is to apply Theorems \ref{First.Technical.Thm}
 and \ref{Second.Technical.Theorem} to the constructions of the previous subsection. This involves studying the various sheaves associated to $W_d$ and $Z_d$, and their restrictions to a general hyperplane. So we start by fixing notation and making a few preliminary remarks. 

Assume as in the previous subsection  that we have data $H$, $\ol{X}_j$, $Z_d$ and $W_d$   adapted to $B$. Set  $\ol{X}  = \ol{X}_n \in \linser{H}$, and for $1 \le j \le n-1$ let $\ol{\ol{X_j}} = \ol{X}_ j \mid \ol{X}$. Denote by 
\[
\ol{L}_d \ , \ \ol{B} \ , \ \ol{H} \ , \  \ol{D}_i
\] 
the restrictions to $\ol{X}$ of the corresponding divisors on $X$, and set 
\[ \ol{Z}_d \ = \ Z_d \mid \ol{X} \ = \ \ol{D}_1 \, \cap \, \ldots \, \cap \ol{D}_c. \]
In order to analyze the restrictions to $\ol{X}$ of $V_d$ and $W_d$, let
\[   V^\pr_d \ = \ \HH{0}{X}{\OO_X(L_d) \otimes \II_{\ol{X}/X} } \ = \ \HH{0}{X}{\OO_X(L_d - H) }
\] and put
\[  \ol{V}_d \ = \ V_d / V^\pr_d \ \underset{  (d \gg 0) }{=} \ \HH{0}{\ol{X}}{\OO_{\ol{X}}(\ol{L}_d)}.
\]
Observe that
\begin{equation}
\label{Dim.vdprime}
v_d^\pr \ =_{\text{def}} \ \dim \, V_d^\pr \ = \ v_d - O(d^{n-1}).  \end{equation}
Similarly, set
\begin{equation}
  W^\pr_d  \ = \ \pi_d(V^\pr_d)\ \ , \ \ \ol{W}_d \ = \ W_d / W_d^\pr \ \ , \ \ \ol{w}_d \ = \ \dim \ol{W}_d,  \end{equation}
and let
\begin{equation}  \label{Two.Ideals.Xbar}
 \ol{J}_d \ \subseteq \ \HH{0}{\ol{X}}{\II_{\ol{Z}_d / \ol{X} } \otimes \OO_{\ol{X}}(\ol{L}_d)}\end{equation}
denote the image of $J_d$ on $\ol{X}$, so that
$ \ol{W_d}= \ol{V}_d / \ol{J}_d. $ 

The next statement shows that our constructions behave well inductively with respect to Definition \ref{Def.Adapted}. 
\begin{lemma} \label{Restrictions.are.Adapted} If the data $H$, $\ol{X}_j$, $Z_d$ and $W_d$ are adapted to $B$ on $X$, then their
  restrictions $\ol{H}$, $\ol{\ol{X}}_j$, $\ol{Z}_d$ and $\ol{W}_d$ are adapted to the divisor $\ol{B}+\ol{H}$ on $\ol{X}$. 
\end{lemma}
\begin{proof}
Note to begin with that
\[   \big( B - K_X \big) | \ol{X} \ \lin \ ( \ol{B} +\ol{H}) - K_{\ol X}  \]
thanks to the adjunction formula. The fact that $\ol{H}$ and $\ol{Z}_d$ satisfy the required properties then  follows at once using the exact sequence $0 \to \OO_X (-H) \to\OO_X \to \OO_{\ol{X}} \to 0$ of sheaves on $X$. Turning to the conditions on $\ol{W}_d$, it is immediate that $\ol{J}_d$ generates the sheaf $\II_{\ol{Z}_d/\ol{X}}\otimes \OO_{\ol{X}}(\ol{L}_d)$. It remains only to show that 
\begin{equation}
\label{Restricted.W.Bdd.Dim.Eqn}
\dim \, \frac{ \HH{0}{\ol{X}}{\II_{\ol{Z}_d / \ol{X} } \otimes \OO_{\ol{X}} (\ol{L}_d)}}{\ol{J}_d}  \ \le \ \ol{a}  \end{equation}
for some integer $\ol{a}$ independent of $d$. For this, 
let
\[ \ol{I}_d  \ = \ \Image \Big( \HH{0}{X}{\II_{Z_d/X}(L_d)} \lra \HH{0}{\ol{X}}{\II_{\ol{Z}_d/\ol{X}}(\ol{L}_d)}\Big).\]
Then on the one hand, 
\[ \dim \,    \ol{I}_d / \ol{J}_d  \ \le \ a \]
thanks to  \eqref{Second.Cond.Wd}. On  the other hand, there is a natural injection
\[ \frac{ \HH{0}{\ol{X}}{\II_{\ol{Z}_d/\ol{X}}(\ol{L}_d)}}{\ol{I}_d} \ \subseteq \  \HH{1}{ {X}}{\II_{ {Z}_d/{X}}( {L}_d -H)}, \]
and the dimension of the group on the right is bounded independent of $d$ by virtue of Lemma \ref{Cohomol.Properties.Lemma}
 (ii). \end{proof}
 
The following statement contains the main inductive step:
\begin{lemma} \label{Main.Inductive.Proposition}
Fix     a divisor $B$, and assume that
  $H$, $Z_d$ and $W_d$ are adapted to $B$.  
  \begin{itemize}
    \item[(i).]  If $q \ge 2$ then for sufficiently large $d$, $W_d$ carries weight $q$ syzygies of $B$ with respect to $L_d$;
    \vskip 10pt
    \item[(ii).] If $q \ge 3$ then for sufficiently large $d$,  $\ol{W_d}$ carries weight $(q-1)$ syzygies of $\ol{B} + \ol{H}$ on $\ol{X}$ with respect to $\ol{L}_d$. 
\end{itemize}   
\end{lemma}

\begin{proof}
The plan of course is  to apply Theorem \ref{First.Technical.Thm}. To this end we consider
 the exact sequence
\begin{equation} \label{Restricted.Secant.Seq}
0 \lra \Sigma_{\ol{W}_d} \lra \ol{W}_d \otimes \OO_{\ol{X}} \lra \OO_{Z_d} (\ol{L}_d) \lra 0 
\end{equation}
of sheaves on $\ol{X}$, and the quotient $ M_{\ol{L}_d} \twoheadrightarrow \Sigma_{\ol{W}_d} $, giving rise to  \begin{equation}
\ol{\sigma}_d : \Lambda^{\ol{w}_d} M_{\ol{L}_d} \lra \II_{\ol{Z}_d/\ol{X}}. \notag 
\end{equation}
Observe that if $q = 2$, so that $\dim Z_d = 0$, then $\ol{Z}_d = \varnothing$, and \eqref{Restricted.Secant.Seq} shows that 
\[ 
\Sigma_{\ol{W}_d} = \ol{W}_d \otimes \OO_{\ol{X}} 
\] 
is a trivial bundle of rank $= \ol{w}_d$. 

For fixed $c = n + 2 -q$ we proceed by induction on $n$ -- or equivalently on $q$ -- starting with   $q = 2$. In this case (as we have just noted) $\ol{Z}_d = \varnothing$, and  we claim that 
the homomorphism 
\begin{equation} \label{H0.Surjectivity.to.Verify}
  \ \HHH{0}{\ol{X}}{\Lambda^{\ol{w}_d} M_{\ol{L}_d}(\ol{B} + \ol{L}_d + \ol{H})} \lra \HHH{0}{\ol{X}}{\OO_{\ol{X}}(\ol{B} + \ol{L}_d + \ol{H})}  
  \end{equation}
determined by $\ol{\sigma}_d$ is surjective for $d \gg 0$. Indeed,  $\ol{\sigma}_d$
arises here as an exterior power of the trivial quotient 
\[
M_{\ol{L}_d}\lra  \Sigma_{\ol{W}_d} \ = \ \ol{W}_d \otimes \OO_{\ol{X}}.
\]
Moreover the fact (Lemma \ref{Restrictions.are.Adapted}) that $\ol{W}_d$ is adapted to $\ol{B} + \ol{H}$ -- in particular, equation \eqref{Restricted.W.Bdd.Dim.Eqn} -- implies that that $\ol{w}_d = \rk (\ol{W}_d)$ is bounded independent of $d$.  So the required surjectivity follows for $d \gg 0$ from Lemma \ref{Surjectivity.H0.Lemma}. 

Now for large $d$ it is automatic that $L_d$ satisfies the vanishings  \eqref{Van.Hypoth.Kosz.from.M}, while    \eqref{Van.of.Hi} was established in Lemma  \ref{Higher.Cohom.Twists.Ideal.Z} (i).
Therefore Theorem \ref{First.Technical.Thm} applies, and having checked  the surjectivity of \eqref{H0.Surjectivity.to.Verify} we deduce that statement (i) of the Lemma holds when $q = 2$. But thanks to Lemma \ref{Restrictions.are.Adapted}, we can then apply this statement to the divisor $\ol{B} + \ol{H}$ on $\ol{X}$  to conclude statement (ii) of the Lemma in the case $q =3$. Theorem \ref{First.Technical.Thm} 
then gives the case $q = 3$ of statement (i), and we continue inductively in this fashion.  
\end{proof}

Theorem \ref{Main.NonVan.Statement.Section.2} is now immediate.
\begin{proof} [Proof of Theorem \ref{Main.NonVan.Statement.Section.2}] We simply invoke Theorem \ref{Second.Technical.Theorem}. In fact, the hypotheses of that statement are satisfied thanks to the surjectivity of \eqref{H0.Surjectivity.to.Verify} (when $q = 2$) and Proposition \ref{Main.Inductive.Proposition} (ii) (when $q \ge 3$), as well as Lemma  \ref{Cohom.Properties.Ideal.Lemma} (i) and the observation that $\ol{v}_d - \ol{w}_d = O(d^{n-1})$. We deduce that if $d \gg 0$, then
 $K_{p,q}(X, B; L_d) \ne 0 $ for all 
\[ w_d + 1 - q \ \le \ p \ \le \ (v^\pr_d + \ol{w_d}) + 1 - q. \] 
The first statement of the theorem  then follows from   the estimates for $w_d$ and $v^\pr_d$ given by Lemma \ref{Estimate.wd.Lemma} and equation \eqref{Dim.vdprime} respectively. If $\HH{i}{X}{\OO_X(B)} = 0$ for $0 < i < n$, then for large $d$ we can apply duality (Proposition \ref{Duality}) to the the non-vanishing just established (and Proposition \ref{Kp1.Proposition} below) to get the second assertion.
 \end{proof}


\section{The groups $K_{p,0}$ and $K_{p,1}$} \label{Kp1.Section}

In this section, we complete the proof of Theorem \ref{Main.Thm.Intro} from the Introduction by analysing syzygies of weight $0$ and $1$. As above, $X$ is a smooth projective variety of dimension $n$, $B$ is a fixed divisor on $X$, and $L_d = dA + P$ where $A$ is ample and $P$ is arbitrary.

To begin with, a  theorem of Green and an argument of Ottaviano--Paoletti yield complete control  over the non-vanishing of $K_{p,0}$: 
\begin{proposition} \label{Kp0.Proposition}
If $d$ is sufficiently large, then
\[   K_{p,0}(X,B; L_d) \ \ne \ 0 \ \Longleftrightarrow \ p \, \le \, r(B),\]
where as usual $r(B) = \hh{0}{X}{\OO_X(B)}-1$.
\end{proposition}
\begin{proof}
We assume in the first place that $d$ is large enough so that $\HH{0}{X}{\OO_X(B -mL_d)} = 0$ for $m \ge 1$, in which case
\[  K_{p,0}(X, B; L_d) \ = \ Z_{p,0}(X,B;L_d) \ = \ \HH{0}{X}{\Lambda^p M_{L_d} \otimes B}   \]
 thanks to Remark \ref{Kp0. Lemma}. The vanishing of this group  when $p \ge \hh{0}{X}{\OO_X(B)}$ was established by Green \cite[Theorem 3.a.1]{Kosz1}. On the other hand, supposing that $p < \hh{0}{X}{\OO_X(B)}$,  choose linearly independent sections
\[  f_0 \, , \, \ldots \, , \, f_p \ \in \ \HH{0}{X}{\OO_X(B)}  \]
and fix (when $d$ is large compared to $B$) a non-zero section $s \in \HH{0}{X}{\OO_X(L_d - B)}$. Inspired by \cite{OP},  consider   the element
\small
 \[    \alpha \ =_{\text{def}} \ \sum_{j =0}^{p} (-1)^j \big( f_0s \wedge \ldots \wedge \widehat{f_js} \wedge \ldots  
\wedge f_ps \big) \otimes f_j \ \in \ \Lambda^p \HH{0}{X}{\OO_X(L_d)} \otimes \HH{0}{X}{\OO_X(B)}. 
 \]
 \normalsize
 Then $\alpha$ is killed by the Koszul differential, and hence gives  a non-zero class in $Z_{p,0}(X, B;L_d)$, as required. 
 \end{proof}
 
By applying the Proposition with $B$ replaced by $K_X - B$, and using duality \cite[2.c.6]{Kosz1}, one finds:
 \begin{corollary}
 If $d$ is sufficiently large, then
$K_{p, n+1}(X,B;L_d) \ne 0$ if and only if \[
r_d - n - r(K_X -B)  \ \le \  p  \ \le \ r_d - n. \qed\]
 \end{corollary}
 
 \begin{remark}({\bf{Non-Cohen-Macaulay modules}}). If $\HH{i}{X}{\OO_X(B)} = 0$ for $0 < i < n$, then the module $R(B;L_d)$ is Cohen-Macaulay for $d \gg0$, and hence
 \[   K_{p,q}(X, B; L_d) \ = \ 0 \ \ \text{for  }   p\, > \,  r_d -n. \]
 We have just seen that when $q = n+1$ this holds for any $B$. On the other hand, if $\HH{q-1}{X} {\OO_X(B)} \ne 0$ for some $2 \le q \le n$, then
 \[ 
 \HH{q-1}{X}{\Lambda^{r_d} M_{L_d} \otimes \OO_X(B+L) } \ = \  \HH{q-1}{X}{\OO_X(B) } \ \ne \ 0,
 \]
 and hence $K_{r_d - (q - 1),q}(X,B;L_d) \ne 0 $ for large $d$. \qed
 \end{remark}
 
 Turning to $K_{p,1}$, we resume the notation of $\S 2$. Thus we choose a suitably positive very ample class $H$, we fix a general divisor $\ol{X} \in \linser{H}$, and set
 \[
 \begin{gathered}V_d \ = \ \HH{0}{X}{\OO_X(L_d)} \\
    V^\pr_d \ = \ \HH{0}{X}{\OO_X(L_d) \otimes \II_{\ol{X}/X} }   
    \ \ , \ \ v_d^\pr \ = \ \dim \, V_d^\pr.
\end{gathered}
 \]
\begin{proposition} \label{Kp1.Proposition}
 If $d$ is sufficiently large, then 
\[ K_{p,1}(X,B;L_d) \ \ne \ 0\]
provided that $  \hh{0}{X}{\OO_X(B+H)} \le p  \le v_d^\pr - 1$. 
 \end{proposition}
 \noi More picturesquely, the Proposition asserts that $K_{p,1}(X,B;L_d) \ne 0$ for \[ O(1) \ \le \ p \ \le\ r_d - O(d^{n-1}).\] 

\begin{proof}
The issue is to produce an element of $\Lambda^{p}V_d \otimes H^0\big(\OO_X(B+L_d)\big)$ representing a non-trivial Koszul cohomology class.
To this end, consider the exact commutative diagram:
\Small
\[
\xymatrix{
0 \ar[r] &\Lambda^{p+1}V_d \otimes H^0\big(\OO_X(B)\big) \ar[r] \ar[d]^\delta &\Lambda^{p+1}V_d \otimes H^0\big(\OO_X(B+H)\big) \ar[r]^{r} \ar[d]^\delta & \Lambda^{p+1}V_d \otimes H^0\big(\OO_{\ol{X}}(\ol{B} + \ol{H})\big) \ar[d]^{\ol{\delta}}\\
0 \ar[r] &\Lambda^{p}V_d \otimes H^0\big(\OO_X(B+L_d)\big) \ar[r]  &\Lambda^{p}V_d \otimes H^0\big(\OO_X(B+H+L_d)\big) \ar[r]   & \Lambda^{p }V_d \otimes H^0\big(\OO_{\ol{X}}(\ol{B} + \ol{H}+ \ol{L_d})\big) .  
}
\]
\normalsize
If  $p + 1 >  \hh{0}{X}{\OO_X(B+H)} $ then the two vertical maps on the left are injective by virtue of Green's lemma \cite[Theorem 3.a.1]{Kosz1}. This being so, a diagram chase as in the snake lemma shows that it suffices to exhibit a  class
\[
0 \ \ne \ \alpha \ \in \ \Lambda^{p+1}V_d \otimes H^0\big(X, \OO_X(B+H)\big)
\]
having the properties:
\[  \ol{\alpha} \, =_\text{def}\, r(\alpha) \, \ne \, 0 \ \ , \ \ \ol{\delta}\ol{\alpha} \, = \, 0 . \]
But for this one can take any element
\[  \alpha \, \in \, \Lambda^{p+1} V^\pr_d \otimes H^0\big(X, \OO_X(B+H)\big) \]
whose image
\[ \ol{\alpha} \, \in \, \Lambda^{p+1} V^\pr_d \otimes H^0\big(\ol{X}, \OO_{\ol{X}}(B+H)\big)  \]
is non-zero, for then automatically $\ol{\delta}{\ol{\alpha}} = 0$.
\end{proof}

\section{Veronese Varieties} \label{Veronese.Vars.Section}
 
 In this section, we consider in more detail the case $X = \PP^n$. By modifying the arguments appearing in \S 2, we prove  effective non-vanishing results. At least in some  some cases (and conjecturally in all), the statements that we obtain are optimal. As a corollary, we establish a non-vanishing theorem for the Schur decomposition of the syzygies of Veronese varieties in characteristic zero.

Turning to details, let $U$ be a vector space of dimension $n + 1$, so that $\PP(U) = \PP^n$ is an $n$-dimensional projective space. We take
\[ L_d   \in   \linser{\OO_{\PP}(d)} \quad , \quad  B \in   \linser{\OO_{\PP}(b)} . \] It will be convenient to denote the corresponding  Koszul group $K_{p,q}(\PP^n, B;L_d)$  by 
\[   K_{p,q}(\PP^n, b;d) \quad \text{or } \quad K_{p,q}(\PP(U), b;d). \]
\begin{theorem} \label{Effective.Statement.Proj.Space}
Fix an index $1 \le q \le  n$, and assume that $b \ge 0$. If $d$ is sufficiently large, then \[ K_{p,q}(\PP^n,b;d) \ \ne \ 0\] for
\small
\[
\binom{d+q}{q} \, - \,  \binom{d-b-1}{q} \, - \, q   \ \le \  p \ \le \ \binom{d +n  }{n} \,  - \, \binom{d+n -q}{n-q} \, + \,  \binom{n+b}{n-q} \, - q-1.
\]
\normalsize
\end{theorem}
\noi 
One can show that in fact the assertion holds as soon as $d \ge b+n +1 $ (Remark \ref{d.vs.b+n+1}).  
 When $q = 0$ or $q=n$, the   bounds on $p$ are the best possible (Remark \ref{Kp0.Kpn.Proj.Space}). It seems reasonable to hope that the statement is optimal in general, and that it holds moreover whenever $d \ge b + q + 1$.   As  noted in the Introduction, Weyman informs us that he independently obtained the case $b = 0$ of the Theorem, at least in characteristic $0$.

Before turning to the proof of the theorem,  we record a representation-theoretic corollary. Specifically, note that $K_{p,q}(\PP(U),b;d)$ is the cohomology of the Koszul-type complex
\[
\small
\ldots \lra \Lambda^{p+1}\big( S^dU\big) \otimes S^{(q-1)d +b}U \lra \Lambda^{p} \big( S^d U\big) \otimes S^{qd +b}U \lra \Lambda^{p-1} \big(S^dU \big)\otimes S^{(q+1)d +b}U \lra \ldots \ \ . 
\normalsize
\]
This shows that $K_{p,q}(\PP(U), b;d)$ is functorial in $U$, and we denote by $\mathbf{K} _{p,q}(b;d)$ the corresponding functor.\footnote{This notation was suggested by Andrew Snowden.} Thus
\[  \mathbf{K}_{p,q}(b;d)\big(U\big) \ = \ K_{p,q}(\PP(U), b;d)  \]
for any vector space $U$. Assuming that we are working over the complex numbers, one can argue as in \cite{Rubei} or \cite{Snowden} that one has a decomposition
\[
 \mathbf{K}_{p,q}(b;d) \ = \ \bigoplus_{\lambda \, \vdash \, (p+q)d + b} M_{\lambda}  \otimes_{\CC} \mathbf{S}_\lambda, \tag{*}
\]
where $\mathbf{S}_{\lambda}$ is the Schur functor associated to the partition $\lambda$, and $M_\lambda = M_{\lambda}(p,q,b;d)$ is a finite-dimensional vector space giving the multiplicity of $\bf{S}_\lambda$ in $ \mathbf{K}_{p,q}(b;d)$.  It is established by Rubei in \cite{Rubei} that as soon as
\[  K_{p,q}(\PP(U_0),b;d) \ \ne \ 0 \]
for some vector space $U_0$ of dimension $p+1$ or greater, then $K_{p,q}(U, b;d) \ne 0$ for every vector space of dimension $\ge p+1$.\footnote{The point here is that all the  $\mathbf{S}_\lambda$ appearing non-trivially in $K_{p,q}(\PP(U_0),b;d)$ are indexed by Young diagrams   having at most $p+1$ rows. And for any such $\lambda$, $\mathbf{S}_\lambda(U_0) \ne 0$ once $\dim U_0 \ge p + 1$.}  We will say in this case that $\mathbf{K}_{p,q}(b;d) \ne \mathbf{0}.$

Working still over the complex numbers, one has:
\begin{corollary} \label{Veronese.Stable.Syzygies.Cor}
Fix $q\ge 1$, and integers $b \ge 0$ and $d \ge q+b+1$. Assume that   \[
p \ \ge \ \binom{d+q}{q} \, - \,  \binom{d-b-1}{q} \, - \, q.\]
Then \[ \mathbf{K}_{p,q}(b;d) \ \ne \ \mathbf{0}. \]  \end{corollary}
\begin{proof}[Proof of Corollary \ref{Veronese.Stable.Syzygies.Cor}]
Assume that $p$ satisfies the stated inequality, and suppose for the moment that one knew (as one expects) that Theorem \ref{Effective.Statement.Proj.Space}
 holds for every $d \ge b + q + 1$. If we take $n+1 = \dim U$ to be sufficiently large, then $p$ will also satisfy the upper bound appearing in the statement of the Theorem. This would imply that $K_{p,q}(\PP(U),b;d) \ne 0$, and hence that $\mathbf{K}_{p,q}(d) \ne \mathbf{0}$, as required. Unfortunately, we haven't established that \ref{Effective.Statement.Proj.Space} is valid in the expected range of $d$. However what is required for the Corollary is the non-vanishing of $K_{p,q}$ with the stated lower bound on $p$, and an upper bound that goes to infinity with $n$. This is provided by Propositions \ref{Weaker.Non.Van.Veronese.Syz} and \ref{Kp1.Veronese.Prop} below.
 \end{proof}

The proof of Theorem \ref{Effective.Statement.Proj.Space} proceeds in three steps. First, assuming that $2 
\le q \le n$, we use a variant of the arguments from \S 2 to prove a result with the stated lower bound on $p$ but a somewhat weaker upper bound. For this we exploit the negativity of the canonical bundle on $\PP^n$  to improve the constructions occurring in \S 2.  We then sketch how to extend the same statement to the case $q = 1$. Finally, we 
use duality to deduce Theorem  \ref{Effective.Statement.Proj.Space} as stated.
\begin{proposition} \label{Weaker.Non.Van.Veronese.Syz}
Fix $2 \le q \le n$, $b \ge 0$ and $d \ge b + q + 1$.  Then 
\[ K_{p,q}(\PP^n,b;d) \ \ne  \ 0\] for
\small
\[
\binom{d+q}{q} \, - \,  \binom{d-b-1}{q} \, - \, q   \ \le \  p \ \le \ \binom{d +n -1 }{n} \, + \,  \binom{d+q-1}{q-1} \, - \, \binom{d-b-2}{q-1} - q.
\]
\normalsize
\end{proposition}
\begin{proof}
Working for the moment on an $m$-dimensional projective space $\PP^m$, fix $2 \le c \le m$ and an integer $a \ge 0$, and set $s = m -c + 2$.  We assume that $d \ge a+s+1$. Choose general divisors
\begin{equation} \label{Def.of.Di.on.Pm.Eqn}
\begin{gathered}
D_1 \ , \ D_2 \ , \ \ldots \ , \ D_{c-2} \ \in \ \linser{ \OO_{\PP^m}(1) } \\ 
D_{c-1} \ \in \ \linser{\OO_{\PP^m}(a+s+1) } 
\ \ , \ \ 
D_c \ \in \ \linser{ \OO_{\PP^m}(d) },\end{gathered} \notag
\end{equation}
and  let 
\begin{equation} \label{Def.of.ZonP.Eqn}
 Z_{d,m,a}\ = \ D_1 \, \cap \, \ldots \, \cap D_c  \notag
\end{equation}
be the complete intersection of the $D_i$. Thus $Z=Z_{d,m,a}$ is a smooth variety of dimension $m -c = s-2$, and one has
\[
\HH{s-1}{\PP^m}{ \II_{Z/\PP^m}(d+a)} \ \ne \ 0 \quad , \quad \HH{s-1}{\PP^m}{ \II_{Z/\PP^m}(d+a+1) }\ = \ 0.
\]
Moreover, $\II_{Z}(d)$ is globally generated. 
Put
\[   \ W_{d,m,a} \ = \ \Image \Big( \HH{0}{\PP^m}{\OO_{\PP^m}(d)} \lra \HH{0}{Z }{\OO_{Z }(d)}\Big).
\]
If $\dim Z > 0$  then  the map appearing on the right is surjective, and one finds that
\[    w_{d,m,a} \ =_{\text{def}} \ \dim W_{d,m,a} \ = \ \binom{d+s}{s} - \binom{d-a-1}{s} - 1 \]
for every $s \ge 2$.\footnote{For this and subsequent calculations, it is useful to observe that one can view $Z_{d,m,a}$ as being a complete intersection of type $(a+s+1,d)$ in the projective space $\PP^{s} = D_1 \cap \ldots \cap D_{c-2}$.}
Writing $M_{d,m}$ for the bundle \eqref{Seq.Def.M_V} on $\PP^m$ associated to $\OO_{\PP^m}(d)$, we get as in \S 1 a homomorphism
\[  
\sigma_{d,m,a} \, : \,  \bigwedge^{w_{d,m,a}}M_{d,m} \lra \II_{Z_{d,m,a}/\PP^m}.
\]

Now consider  a general hyperplane $ \PP^{m-1} \subseteq \PP^m.$ Then the intersection \[
\ol{Z}_{d,m,a} \ =_\text{def} \ Z_{d,m,a} \, \cap \, \PP^{m-1}\] can be viewed as realizing \[ Z_{d, m-1, a+1}\  \subseteq \ \PP^{m-1}.\]  Note that in this identification $c$   unchanged, and hence $s$ is decreased by $1$. Similarly,  with notation as in \eqref{Restrict.VW}, the vector space $W_{d,m,a}$ restricts on $\PP^{m-1}$ to 
\[
\ol{W}_{d,m,a} \ = \ W_{d, m-1,a+1} 
\]
provided that $c < m$. If $m = c$, so that $s = 2$, then $\dim Z_{d,m,a} = 0$, and
\[   \ol{W}_{d,m,a} \ = \ \ker \Big ( \HH{1}{\PP^m}{\II_{Z_{d,m,a}}(d-1)} \lra \HH{1}{\PP^m}{\II_{Z_{d,m,a}}(d)} \Big).
\]
One finds by a calculation on $\PP^2$ that in this case
\[  \ol{w}_{d,m,a} \ =_{\text{def}} \ \dim \ol{W}_{d,m,a} \ = \ a+2.
\]

Turning to the proof of the statement in the Proposition,   we start by applying this construction on $X = \PP^n$ with  $a = b$ and $c = n + 2 - q$ for the given index $2 \le q \le n$. We wish to show in this case   that $W_{d,n,b}$ satisfies the hypotheses of Theorems \ref{First.Technical.Thm} and \ref{Second.Technical.Theorem}. 
For this the plan is repeatedly to apply Theorem \ref{First.Technical.Thm}, and to argue by descending induction on $i$ ($0 \le i \le q-2$) that the homomorphisms 
\[
\HH{q-1-i}{\PP^{n-i}}{\Lambda^{w_{d,n-i,b+i}}M_{d,n-i} (b+d+i) }\lra \HH{q-1-i}{\PP^{n-i}}{ \II_{Z_{d,n-i,b+i}/\PP^m}(d+b+i)}
\]
determined by $\sigma_{d,n-i,b+i}$ are surjective. In view of the remarks in the previous paragraphs, the only issue is to  get the induction started when $i = q-2$. Here one has to prove the surjectivity of the map
\small
\[
\HH{0}{\PP^{n+1-q}}{\Lambda^{\ol{w}_{d, n+2-q,b+q-2}}M_{d,n+1-q}(b+d+q-1)} \lra \HH{0}{\PP^{n+1-q}}{\OO_{\PP^{n+1-q}}(b+d+q-1)}
\]
\normalsize
arising from 
\[
\ol{\rho}_{d,n+1-q,b+q-2}\, : \, M_{d,n+1-q} \lra \ol{W}_{d,n+2-q,b+q-2,} \otimes_{k} \OO_{\PP^{n+1-q}}.
 \]
The term on the right is a trivial vector bundle of rank $b+q$. Recalling that  $M_{d,n+1-q}(1)$ is globally generated, the required surjectivity follows via an Eagon--Northcott complex as in the proof of Proposition \ref{Surjectivity.H0.Lemma}.
 
 We may now apply Theorem \ref{Second.Technical.Theorem}. In the case at hand, it gives the non-vanishing of $K_{p,q}(\PP^n, b;d)$ in the range
 \[
 w_{d,n,b} + 1 - q \ \le \ p \ \le\   \hh{0}{\PP^n}{\OO_{\PP^n}(d-1)} \, + \, \ol{w}_{d,n,b} \, + \, 1 \, - \,  q.
 \]
 Noting that 
$
 \ol{w}_{d,n,b}= w_{d,n-1,b+1} = \binom{d-q-1}{q-1} - \binom{d-b-2}{q-1} - 1,
$
 the result follows. 
 \end{proof}
 
 We next indicate the extension of the previous result to the case $q =1$. 
 \begin{proposition} \label{Kp1.Veronese.Prop}
 Assume that $b \ge 0$ and that $d \ge b+2$. Then $K_{p,1}(\PP^n,b;d) \ne 0$ for
 \[
 b \, + \, 1 \ \le \ p \ \le \ \binom{d+n-1}{n} \, - \, 1.
 \]
 \end{proposition}
 \begin{proof}[Sketch of Proof] Since $\HH{1}{\PP^n}{\OO_{\PP^n}(b))} = 0$, it is enough (by Proposition \ref{Non.Van.Wedge.p+1.Gives.Koszul}) to show  that 
 \[
 \HH{1}{\PP^n}{\Lambda^{p+1}M_d \otimes \OO_{\PP^n}(b) } \ \ne \ 0 
 \]
 for $p$ in the stated range, where $M_d$ denotes the vector bundle \eqref{Seq.Def.M_V}  associated to $\OO_{\PP^n}(d)$. For this, we take $Z \subseteq \PP^n$ to consist of $b+2$ collinear points, so that 
  \[
\HH{1}{\PP^n}{ \II_{Z/\PP^n}(b)} \ \ne \ 0 \quad , \quad \HH{1}{\PP^n}{ \II_{Z/\PP^n}(b+1) }\ = \ 0. \]
We apply the constructions of \S 2 (or the previous proof), with 
 \begin{align*}
 W \ &= \ \Image \  \Big( \HH{0}{\PP^n}{\OO_{\PP^n}(d)} \lra \HH{0}{Z }{\OO_{Z }(d)}\Big) \\ &= \ \HH{0}{Z}{\OO_Z(d)}. 
 \end{align*}
 As above,  this gives rise to a mapping
 \[
 \HH{1}{\PP^n}{ \Lambda^{w} M_d \otimes \OO_{\PP^n}(b) } \lra \HH{1}{\PP^n}{\II_{Z/\PP^n}(b)}, \tag{*}
 \]
where $w = \dim W = b+2$, whose surjectivity we wish to establish. To this end, observe that with notation as in \eqref{Restrict.VW}, one has 
\[
W^\pr \ = \ W \ \ , \ \ \ol{W} \ = \ 0.
\]
This being so, the homomorphism $\ol{\sigma}$ appearing in  Theorem \ref{First.Technical.Thm} is automatically surjective and then the argument in that proof applies without further ado (cf.\ Remark \ref{Wbar.Eqs.Zero.Remark}). This shows that $K_{p,1}(\PP^n,b;d) \ne 0$ for $p = b+1$. But now as in the proof of Theorem \ref{Second.Technical.Theorem}, we can replace $W$ by a larger quotient $U$ of $\HH{0}{\PP^n}{\OO_{\PP^n}(d)}$, having the same properties, with $u = \dim U$ taking any value $b+1 \le u \le \binom {n+d-1}{n}$. The proposition follows.  
 \end{proof}

\begin{remark}[$K_{p,0}$ and $K_{p,n}$] \label{Kp0.Kpn.Proj.Space}
It follows from (the proof of) Proposition \ref{Kp0.Proposition} that if $d \ge b+1$, then
\[ K_{p,0}(\PP^n,b;d) \ \ne \ 0 \] if and only if
\[
0 \ \le \ p \ \le \ \binom{b+n}{n} - 1.
\]
By duality (Proposition \ref{Duality}) this implies that if $d \ge b + n +1$, then
\[
K_{p,n}(\PP^n,b;d) \ \ne \ 0 
\]
if and only if
\[
\binom{d+n}{n} \, - \, \binom{d-b-1}{n} \, - \, n \ \le \ p \ \le \ \binom{d+n}{n} \, - \, n  - 1
\]
This shows that the inequality in Theorem \ref{Effective.Statement.Proj.Space} is optimal when $q = 0$ and $q =n$. 
\end{remark}

\begin{proof}[Proof of Theorem \ref{Effective.Statement.Proj.Space}] The case $q = n$ having been established in the previous Remark, we assume that $1 \le q \le n-1$. By Proposition \ref{Duality}, $K_{p,q}(\PP^n,b;d)$ is dual to $K_{p^\pr, q^\pr}(\PP^n, b^\pr;d)$ where
\[
p^\pr \ = \ r_d \, - \, p \, - \, n \ \ , \ \ q^\pr \ = \ n \, - \, q \ \ , \ \ b^\pr \  = \ d \, - \, n \, - \, 	1 \, - \, b.
\]
Applying Propositions \ref{Weaker.Non.Van.Veronese.Syz} and \ref{Kp1.Veronese.Prop} to the latter group, one deduces (after a computation) that $K_{p,q}(\PP^n,b;d) \ne 0$ when
\begin{equation}  \label{Big.Eqn1}
\begin{aligned}
p \ &\ge \binom{d+n-1}{n-1}   \, - \, \binom{d+n-q-1}{n-q-1} \, + \, \binom{n+b-1}{n-q-1} \, - \, q \, - \, 1   \\
 p \ &\le  \ \binom{d +n  }{n} \,  - \, \binom{d+n -q}{n-q} \, + \,  \binom{n+b}{n-q} \, - q-1 . 
\end{aligned}
\end{equation} The second inequality is exactly the upper bound on $p$ appearing in Theorem \ref{Effective.Statement.Proj.Space}. So it is enough to show that the right-hand side of the first inequality is $\le$ the upper bounds appearing in the statements of Propositions \ref{Weaker.Non.Van.Veronese.Syz} and \ref{Kp1.Veronese.Prop}. But this is clear when $d \gg 0$ since the quantity appearing above  is $O(d^{n-1})$ while the expressions in \ref{Weaker.Non.Van.Veronese.Syz} and \ref{Kp1.Veronese.Prop} are $O(d^n)$. 
\end{proof}

\begin{remark} \label{d.vs.b+n+1}
With some effort, one can show that the right-hand side of the first inquality in  \eqref{Big.Eqn1} is majorized by the upper bounds in \ref{Weaker.Non.Van.Veronese.Syz} and \ref{Kp1.Veronese.Prop}  as soon as $d \ge b+ n+1$. It follows that the statement of Theorem \ref{Effective.Statement.Proj.Space} holds in fact for every $d \ge b+ n+1$. However we will not reproduce the calculations here. In any event, we expect that one should only need $d \ge b + q + 1$ in the Theorem. \end{remark}

%
%
%
%
\section{Conjectures and Open Questions} \label{Conjs.Open.Problems.Section}

In this section, we present some problems and conjectures.\footnote{Since the paper was originally written, there has been some further work on some of the questions posed here. We give pointers to these new developments in footnotes.}  Unless otherwise stated, we keep notation as in the body of the paper. Thus $X$ is a smooth projective variety of dimension $n$, 
\[  L_d \ = \ dA + P \]
where $A$ is an ample   and $P$ is an arbitrary divisor, and $B$ is some fixed divisor in whose syzygies we are interested.
As before, $r_d = r(L_d)$. 
 
To begin with, it would be extremely interesting to know whether the lower bounds in our non-vanishing statements have the best possible shape. Based largely on optimism, we hope that they do:
\begin{conjecture} \label{Vanishing.Kpq.Conjecture}
 Fix $2 \le q \le n$. 
One can find a constant $C > 0$ with the property that if $d$ is sufficiently large then 
\[   K_{p,q}(X, B; L_d ) \ = \ 0 \ \ \text{for } \ \ p \ \le \ C \cdot d^{q-1}.\]
\end{conjecture}
 \noi When $q = 2$ and $k = \CC$, this follows for instance from the results of \cite{SAD}. The one other piece of evidence arises when  $B - K_X$ is represented by a non-zero effective divisor $D$. In this case $H^n(X,B) = 0$, and one finds as in Proposition \ref{Duality} that $K_{p,n}(X,B;L_d)$ is dual to $K_{r_d -p-n,0}(X, L_d-D; L_d)$. But by Proposition \ref{Kp0.Proposition}, this group vanishes when
 \[
 p \ \le \ h^0(L_d) \, - \, h^0(L_d - D) \, - \, n \ = \ O(d^{n-1}).
 \]
 However the general case of the Conjecture seems quite challenging. 
 
 As we saw in Proposition \ref{Kp0.Proposition}, once $d$ is sufficiently large  the vanishing or non-vanishing of $K_{p,0}(X,B;L_d)$ does not depend on $d$. This suggests
 \begin{problem} \label{Small.Syzygies.Question}
 For small $p$ -- for instance, $p \ll h^0(B)$ -- is the set
 \[
     \{  \ p \ \big | \  K_{p,1}(X,B;L_d) \ne 0 \,  \}
 \] 
 independent of $d$ provided that $d$ is large? If so, is there a geometric explanation -- at least under suitable hypotheses on $B$ -- for which of  these groups are non-vanishing?\footnote{Note that if  $q \ge 2$, then it follows for instance from \cite{SAD} that $K_{p,q}(X, B; L_d) = 0$ for $d \ge O(p)$.}
 \end{problem}
 \noi As a very simple example of the sort of thing one might expect, consider the case $p=0$. Then 
 $ K_{0,0} (X, B; L_d)  \ne 0$   if and only if $ h^0(B) \ne 0$, and  if $d$ is sufficiently large, then
\[
 K_{0,1} (X, B; L_d) \ \ne  \ 0   \quad \Longleftrightarrow \quad \linser{B} \ \text{has basepoints}. 
\]
 (In fact, for $d \gg 0$ the maps
 \[
\HH{0}{X}{\OO_X(B)} \otimes \HH{0}{X}{\OO_X(mL_d)}   \lra \HH{0}{X}{\OO_X(B+ mL_d)}
 \]
 are surjective for every $m > 0$ if and only if $\OO_X(B)$ is globally generated.) When $X$ is a curve and $B=K_X$ a precise conjecture along the lines of Problem \ref{Small.Syzygies.Question} was proposed in \cite{GL3}, and   established in many cases by Aprodu and Voisin \cite{AproduVoisin}, \cite{Aprodu}. But even for curves one could ask whether there are statements for other divisors $B$. 
  
  It would be very interesting to understand better the betti numbers of the resolution of $R(B;L_d)$ as $d$ grows.
\begin{problem} \label{DimKpq}
What can one say about the asymptotic behavior of the dimensions 
\[ \dim  K_{p,q}(X,B;L_d)\] as functions of $p$, $q$ and $d$?
\end{problem}
\noi Even in the well-studied case of curves, it's not clear to us what to expect here.\footnote{Some results and conjectures about this problem will appear in a forthcoming paper of Erman, Lee and the authors.}  On a fixed projective space, the work \cite{ES} of Eisenbud--Schreyer on the Boij--Soderberg conjecture gives an overall picture of the betti tables that can arise. As Schreyer and Erman point out, this suggests:
\begin{problem}
As $d$ grows, can one say anything about where the betti tables of the modules $R(B; L_d)$ sit in the Boij--Soderberg cone?
\end{problem}
\noi To begin with,   one would have to find a good way to compare the cones in question as $d$ varies.

Our effective results for Veronese embeddings of projective space raise several questions. Perhaps the most interesting is:
\begin{conjecture}
Fix $1 \le q \le n$, $b \ge 0$, and $d \ge b+ q + 1$. If $p$ lies outside the bounds appearing in Theorem \ref{Effective.Statement.Proj.Space}, then 
\[  K_{p,q}(\PP^n, b;d) \ = \ 0. \] 
\end{conjecture}
 \noi Of course this includes as a special case:
 \begin{conjecture} In the situation of Corollary \ref{Veronese.Stable.Syzygies.Cor}, if 
\[ 
p \ < \ \binom{d+q}{q} \, - \, \binom{d-b-1}{q} \, - \, q,
\] then
$\mathbf{K}_{p,q}(b;d) = \mathbf{0}$. \end{conjecture}
  \noi 
These are probably  very difficult.\footnote{For example, even the conjecture of Ottaviani and Paoletti that $\OO_{\PP^n}(d)$ satisfies property $N_{3d-3}$ for every $n \ge 3$ is open.} So it would already be very encouraging to establish
 Conjecture \ref{Vanishing.Kpq.Conjecture} for $X = \PP^n$: 
 \begin{problem}
 Prove that  
  $\mathbf{K}_{p,q}(b;d) = \mathbf{0}$ when 
 $ p   \le  O( d^{q-1})$.
 \end{problem}
 
  In another direction, one expects that the groups $K_{p,q}(\PP^n,b;d)$ should become quite complicated as the parameters grow. In an attempt to give concrete meaning to this philsophy in characteristic zero, we propose:
  \begin{problem} \label{Many.Schur.Functors}
  Fix $1 \le q \le n$. Prove that if
  \[  p \ \gg \ \binom{d+q}{q} - \binom{d -b-1}{q} -q, \]
  then many different Schur functors $\mathbf{S}_\lambda$ appear non-trivially in $\mathbf{K}_{p,q}(b;d)$. 
  \footnote{Fulger and Zhou have obtained some results in this direction.}
  \end{problem}
 
 It would also be interesting to extend the results for projective space to other ``concrete" varieties.
 \begin{problem} \label{Effective.Statement}
 Prove effective non-vanishing statements for $K_{p,q}(X, B; L_d)$ when $X$ is a rational homogeneous space $G/P$, or a toric variety.
 \end{problem}
 \noi This might be interesting already for Grassmannians.\footnote{Zhou has established effective statements on an arbitrary smooth complex variety provided that  $B = K_X + (n+1)A$ is of adjoint type. His results reduce to Theorem \ref{Effective.Statement.Proj.Space}
 when $X = \PP^n$. 
}

 There has been considerable interest in recent years in extending to the multi-graded case some of the basic algebraic facts concerning linear series on algebraic varieties, eg.\ \cite{HSS}, \cite{MS}, \cite{SVW}. So it is natural to ask whether the results of the present paper admit such an extension. For example, consider ample divisors 
  $A_1, \ldots , A_m$  on a smooth projective variety $X$, and fix arbitrary $P_1, \ldots , P_m$. Given $\mathbf{d}= (d_1, \ldots, d_m)$, with the $d_i$ large, we get an embedding
   \[   X \subset \PP^{  \mathbf{ r_d}} \ =_{\text{def}} \ \PP\big((H^0(P_1+d_1A_1)\big) \times \ldots \times \PP\big( H^0(P_m+d_m A_m) \big). \]
   The multi-homogeneous ideal of $X$ in $\PP^{  \mathbf{ r_d}} $ admits a multi-graded resolution.
\begin{problem}
What can one say about the asymptotic behavior of this resolution as the $d_i$ become very large?
 \end{problem}
 \noi The question is whether there is some interesting multi-graded behavior that does not follow formally from the singly-graded case.\footnote{One might want to assume that the numerical classes of the $A_i$ are linearly independent in $N^1(X)_\RR$.} As far as we know, even the ``classical" results such as Green's theorem for curves and its outgrowths have not been studied in the multi-graded setting.
  
In another direction, as we noted in Remark \ref{Zhou.Remark},  Zhou \cite{Zhou} has shown that one can remove the non-singularity hypothesis from the main non-vanishing theorem, as well as relaxing the conditions on $B$. So the basic asymptotic picture we obtain seems to hold quite generally. This suggests: \begin{problem}
Is there a purely algebraic statement that underlies, or runs parallel to, our results?
 \end{problem}
 
 Finally, we close with  the somewhat vague
 \begin{problem}
 Can one say anything about the   structural properties of the syzygies of $R(B;L_d)$ as $d$ grows?
 \end{problem}
 \noi The fact that generators appear in many degrees does not preclude the possibility that the resolution somehow acquires a relatively simple structure. For example, although it runs in a somewhat different direction, the very interesting work of Snowden \cite{Snowden} shows that there is some   subtle and unexpected finiteness in the syzygies of Segre varieties as one varies the number and dimensions of the factors.  It would be very interesting to know whether something analogous happens in our setting as the positivity of the embedding grows.

 %
 %
 %
 %


\begin{thebibliography}{EMS}
 \setlength{\parskip}{3pt}
 
 
 \bibitem{Aprodu} Marian Aprodu, Green--Lazarsfeld gonality conjecture for a generic curve of odd genus, \textit{Int. Math. Res. Not.} \textbf{63} (2004),3409--3416.
 
 
\bibitem{AproduVoisin} Marian Aprodu and Claire Voisin, Green--Lazarsfeld's conjecture for generic curves of large gonality, \textit{C. R. Math. Acad. Sci. Paris} \textbf{336} (2003), 335-339.

\bibitem{Bombieri} Enrico Bombieri, Canonical models of surfaces of general type, \textit{Publ. Math. IHES} \textbf{42} (1973), 171--219.

 \bibitem{BCR} Winfried Bruns, Aldo Conca, Tim R\"omer, Koszul homology and syzygies of Veronese subalgebras, \textit{Math. ANN.} \textbf{351} (2011), 761--779.
  
 \bibitem{BCR2} Winfried Bruns, Aldo Conca, Tim R\"omer, Koszul cycles, in \textit{Combinatorial Aspects of Commutative Algebra and Algebraic Geometry. Proceedings of the Abel Aymposium 2009}, 17--32.
 
 \bibitem{Castelnuovo} Guido Castelnuovo, Sui multipli di uni serie lineare di gruppi di punti apparetmenente as una curva algebrica, \textit{Rend. Circ. Mat. Palermo} \textbf{7} (1893), 99--119.
 
 \bibitem{Catanese} Fabrizio Catanese, Babbage's conjecture, contact of surfaces, symmetric determinantal varieties and applications, \textit{Invent. Math.} \textbf{63} (1981), 433--465. 
 
 \bibitem{Catanese2} Fabrizio Catanese, Commutative algebra methods and equations of regular surfaces, in \textit{Algebraic Geometry, Bucharest 1982}, Lecture Notes in Math. \textbf{1056}, Springer, Berlin, 1983, 30 -- 50. 
 
 \bibitem{SAD}
Lawrence Ein and Robert Lazarsfeld, Syzygies and Koszul cohomology of smooth projective varieties of arbitrary dimension, \textit{Invent. Math.} {\bf 111} (1993), 51--67.

\bibitem{Eisenbud} David Eisenbud, \textit{The Geometry of Syzygies}, Graduate Texts in Math. \textbf{229}, Springer, 2005.

\bibitem{EGHP}
David Eisenbud, Mark Green, Klaus Hulek and Sorin Popescu, Restricting linear syzygies: algebra and geometry, \textit{Compos. Math.} \textbf{141} (2005), 1460 --1478.

\bibitem{ES}
David Eisenbud and Frank Schreyer, Betti numbers of graded modules and cohomology of vector bundles, \textit{J. Amer. Math. Soc.} \textbf{22}  (2009), 859--888.

\bibitem{Kosz1}
Mark Green, Koszul cohomology and the geometry of projective varieties, J. Diff. Geom. {\bf 19} (1984), 125--171.

\bibitem{Kosz2}
Mark Green, Koszul cohomology and the geometry of projective varieties, II,  J. Diff. Geom. {\bf 20} (1984), 279--289.


\bibitem{GL1} Mark Green and Robert Lazarsfeld, A simple proof of Petri's theorem on canonical curves, in \textit{Geometry Today $($Rome 1984$)$}, Progr. Math. \textbf{60}, Birkhauser, 1985.

\bibitem{GL2}
Mark Green and Robert Lazarsfeld,  Some results on the syzygies of finite sets and algebraic curves, \textit{Compos. Math.} {\bf 67} (1988), 301--314.

\bibitem{GL3}
Mark Green and Robert Lazarsfeld,  On the projective normality of complete linear series on an algebraic curve, \textit{Invent. Math.} {\bf 83} (1985), 73 -- 90.

\bibitem{HSS} Milena Hering, Hal Schenck and Gregory Smith, Syzygies, multigraded regularity and toric varieties, \textit{Compos. Math.} \textbf{142} (2006), 1499--1506.

\bibitem{VBT} Robert Lazarsfeld, A sampling of vector bundle techniques in the study of linear series, in \textit{Lectures on Riemann Surfaces}, World Sci. Publ. 1989, 500--559.

\bibitem{MS} Diane Maclagan and Gregory Smith, Multigraded Castelnouvo--Mumford regularity, \textit{J. Reine. Angew. Math.} \textbf{571} (2004), 179--212.

\bibitem{Mumford1} David Mumford, On the equations defining abelian varieties, \textit{Invent. Math.} {\bf 1} (1966), 287--354.

\bibitem{Mumford2} David Mumford, Varieties defined by quadratic equations, Corso CIME 1969, in \textit{Questions on algebraic varieties}, Rome (1970), 30 -- 100. 

\bibitem{OP} Giorgio Ottaviani and Rafaella Paoletti, Syzygies of Veronese embeddings, \textit{Compos. Math.} {\bf 125} (2001), 31--37.

\bibitem{Pareschi} Giuseppe Pareschi, Syzygies of abelian varieties, \textit{Journal of the AMS} \textbf{13} (2000), 651-- 664.

\bibitem{Rubei} 
Elena Rubei, A result on resolutions of Veronese embeddings, \textit{Ann. Univ. Ferrar Sez. VII} {\bf 50} (2004), 151--165.

\bibitem{Schreyer} Frank Schreyer, Syzygies of canonical curves and special linear series, \textit{Math. Ann.} \textbf{275} (1986), 105--137. 

\bibitem{SVW} Jessica Sidman, Adam Van Tuyl, Haohao Wang, Multigraded regularity: coarsenings and resolutions, \textit{J. Algebra} \textbf{301} (2006), 703--727.


\bibitem{Snowden} Andrew Snowden, Syzygies of Segre varieties, to appear.


\bibitem{Zhou} Xin Zhou, thesis in preparation.
\end{thebibliography}
 \end{document}